\documentclass[12pt,reqno,a4paper]{amsproc}
\usepackage{amssymb}
\usepackage{amsfonts}
\usepackage{amsmath}
\usepackage{graphicx}
\usepackage{graphics}
\usepackage{esint}
\usepackage{hyperref}
\usepackage{color}
\usepackage{tikz}
\newtheorem{thm}{Theorem}[section]
\newtheorem{lm}{Lemma}[section]
\newtheorem{defn}{Definition}[section]
\newtheorem{pro}{Proposition}[section]
\newtheorem{rem}{Remark}[section]
\newtheorem{cor}{Corollary}[section]
\newtheorem{cla}{Claim}[section]
\newtheorem{exm}{Example}[section]
\newtheorem{conj}{Conjecture}[section]
\newtheorem{que}{Question}[section]

\DeclareMathOperator{\arcsinh}{arcsinh}
\DeclareMathOperator{\dial}{dil}
\DeclareMathOperator{\tr}{tr}

\DeclareMathOperator{\diver}{div}

\DeclareMathOperator{\ric}{Ric}
\DeclareMathOperator{\inj}{inj}
\DeclareMathOperator{\Area}{Area}

\DeclareMathOperator{\vol}{vol}
\DeclareMathOperator{\diam}{diam}
\DeclareMathOperator{\Rm}{Rm}

\makeatletter
\newsavebox\myboxA
\newsavebox\myboxB
\newlength\mylenA

\newcommand*\xoverline[2][0.75]{%
    \sbox{\myboxA}{$\m@th#2$}%
    \setbox\myboxB\null% Phantom box
    \ht\myboxB=\ht\myboxA%
    \dp\myboxB=\dp\myboxA%
    \wd\myboxB=#1\wd\myboxA% Scale phantom
    \sbox\myboxB{$\m@th\overline{\copy\myboxB}$}%  Overlined phantom
    \setlength\mylenA{\the\wd\myboxA}%   calc width diff
    \addtolength\mylenA{-\the\wd\myboxB}%
    \ifdim\wd\myboxB<\wd\myboxA%
    \rlap{\hskip 0.5\mylenA\usebox\myboxB}{\usebox\myboxA}%
    \else
        \hskip -0.5\mylenA\rlap{\usebox\myboxA}{\hskip 0.5\mylenA\usebox\myboxB}%
    \fi}
\makeatother

\begin{document}

\title[]{Total Mean Curvature of the Boundary and Nonnegative Scalar Curvature Fill-ins}

\author{Yuguang Shi}
\address [Yuguang Shi] {Key Laboratory of Pure and Applied Mathematics, School of Mathematical Sciences, Peking University, Beijing, 100871, P.\ R.\ China}
\email{ygshi@math.pku.edu.cn}
\thanks{Yuguang Shi, Guodong Wei  are partially supported by NSFC 11671015 and 11731001.  Wenlong Wang is partially supported by NSFC 11671015, 11701326.}

\author{Wenlong Wang}
\address [Wenlong Wang] {School of Mathematical Sciences and LPMC, Nankai University, Tianjin, 300071, P. R. China}
 \email{wangwl@nankai.edu.cn}

\author{Guodong Wei}
\address [Guodong Wei]  {School of Mathematics (Zhuhai), Sun Yat-sen University, Zhuhai, Guangdong, 519082, P. R. China}
\email{weigd3@mail.sysu.edu.cn}

\renewcommand{\subjclassname}{
 \textup{2010} Mathematics Subject Classification}

\date{June 4, 2019}

\begin{abstract}
In the first part of this paper, we prove the extensibility of an arbitrary boundary metric to a positive scalar curvature  (PSC) metric inside for a compact manifold with boundary, which completely solves an open problem due to Gromov (see Question \ref{extension1}). Then we introduce a fill-in invariant (see Definition \ref{fillininvariant}) and discuss its relationship with the positive mass theorems for asymptotically flat (AF) and asymptotically hyperbolic (AH) manifolds. Moreover, we prove that the positive mass theorem for AH manifolds implies that for AF manifolds. In the end, we give some estimates for the fill-in invariant, which provide some partially affirmative answers to Gromov's conjectures formulated in \cite{Gro19} (see Conjecture \ref{conj0} and Conjecture \ref{conj1} below). 
\end{abstract}

\maketitle
\markboth{Shi Yuguang, Wang Wenlong, Wei Guodong  }{Total mean curvature of the boundary and fill-in of NNSC}

\tableofcontents

\section{Introduction}\label{section1}

One of the fundamental problems in Riemannian geometry is to study the relationship between the geometry of the boundary and interior of a compact manifold with boundary and nonnegative scalar curvature (NNSC). However, people even do not know the  extensibility of an arbitrary boundary metric to a positive scalar curvature  (PSC) metric inside. Indeed, M. Gromov raised the following question.
\begin{que}[\cite{Gro19}, p.31--32]\label{extension1}
Let $X$ be a compact manifold with non-empty boundary $Y=\partial X$ and let $\gamma$ be a smooth Riemannin metric on $Y$. Does $\gamma$ always extend to a Riemannin metric $g$ on $X$ with positive scalar curvature? 
\end{que}

Inspired by the arguments in the proof of Theorem 4 in \cite{SWWZ}, we solve Question \ref{extension1} completely. Namely,

\begin{thm}\label{pscextension1}
Let $X$ be a compact manifold with non-empty boundary $Y=\partial X$. Then any metric $\gamma$ on $Y$ can be extended to a Riemannian metric $g$ on $X$ with positive scalar curvature. 
\end{thm}

Note that there is no mean curvature requirements in Theorem \ref{pscextension1}; as a matter of fact, there is a constant $C$ depending on $(X,\gamma)$ such that we can prescribe any mean curvature $H\leq C$ for a PSC extension of $\gamma$.  But we find that pointwisely sufficient large mean curvature rules out such extensions for certain topological types of boundaries. The topological type of boundary is restricted to the following class.

\begin{defn}
Let $\mathcal C_{n}$ $(n\geq 2)$ be the class of closed $n$-manifolds that can be smoothly embedded into $\mathbf R^{n+1}$. 
\end{defn}

By definition, any manifold in $\mathcal C_n$ is orientable and null-cobordant. Actually, $\mathcal C_n$ contains certain nontrivial differential topological types of closed manifolds. For example, all orientable closed surfaces lie in $\mathcal C_{2}$; $\mathbf S^{n_1}\times \mathbf S^{n_2}\times\cdots\times \mathbf S^{n_k}\in \mathcal C_{n_1+n_2+\cdots+n_k}$. 

To state in a more concise way, we introduce the {\it fill-in} notion, which is not new. A triple of (generalized) Bartnik data $(\Sigma,\gamma,H)$ consists of an oriented closed null-cobordant Riemannian manifold $(\Sigma,\gamma)$ and a given smooth function $H$ on $\Sigma$. We say a compact Riemannian manifold $(\Omega, g)$ is a fill-in of $(\Sigma,\gamma,H)$, if there is an isometry $X: (\Sigma, \gamma)\mapsto (\partial \Omega,  g|_{\partial \Omega})$ such that $H=H_g \circ X$ on $\Sigma$, where $H_g$ is the mean curvature of $\partial \Omega$ in $(\Omega, g)$ with respect to the outward normal. If in addition $R_g\geq 0$, we say  $(\Omega, g)$ is a NNSC fill-in of $(\Sigma,\gamma,H)$. Similarly, one can define fill-ins with scalar curvature bounded from below. 

\begin{thm}\label{finiteinfmeancurpsc1}
Suppose $\Sigma\in\mathcal C_{n}$ $(2\leq n\leq 6)$. For any metric $\gamma$ on $\Sigma$, there is a constant $h_0(\Sigma,\gamma)$ depending only on $(\Sigma,\gamma)$ such that for any smooth function $H$ on $\Sigma$ satisfying $H>h_0(\Sigma,\gamma)$, $(\Sigma,\gamma,H)$ admits no fill-in of nonnegative scalar curvature.
\end{thm}

Theorem \ref{finiteinfmeancurpsc1} gives a partial answer to the Problem $\mathbf A$ and $\mathbf {A_1}$ of Problem $\mathbf B$ summarized by Gromov in \cite{Gro18}. Also in \cite{Gro18}, Gromov himself gave a sharp upper bound for the infimum of the mean curvature of NNSC {\it spin} fill-ins. Very recently, when the first edition of this paper was under revision, P. Miao \cite{Miao20} improved Theorem \ref{finiteinfmeancurpsc1} completely by dropping additional topological restrictions on $\Sigma$. His argument is an ingenious use of Theorem \ref{pscextension1} and Schoen-Yau's results on closed manifolds \cite{SY79MM}.

Not only the pointwise mean curvature but also the total mean curvature plays an important role in such fill-in problems.  The following conjecture due to Gromov (see p.23 in \cite{Gro19}) is on the upper bound of the total mean curvature of a fill-in with scalar curvature bounded from below, namely

\begin{conj}\label{conj0}
Let $\sigma$ be a given constant, $(X,g)$ be a compact Riemannian manifold with boundary and scalar curvature $R_g\geq \sigma$. Then the integral of the mean curvature $H$ of the boundary $Y=\partial X$ (with respect to the outward normal) is bounded by 
$$
\int_Y H\,d\mu \leq C,
$$
where $C$ is a constant depending only on $(Y,g|_Y)$ and $\sigma$.
\end{conj}

There are several interesting works supporting this conjecture, for instance, \cite{ST02}, \cite{ST07}, \cite{MM}, \cite{MMT}. In particular, Conjecture \ref{conj0} was verified when $Y$ is a topological $\mathbf{S}^2$ and $H>0$ in \cite{MM}. In \cite{SWWZ}, the authors confirmed Conjecture \ref{conj0} for higher dimensional spheres with metrics PSC-isotopic to the standard one and $H>0$. See also \cite{J,JMT} for some delicate results on mean curvature function but closely related to the total mean curvature for NNSC fill-ins. In this paper, we show that Conjecture \ref{conj0} holds when $X$ is spin, $Y$ is diffeomorphic to a standard sphere and $H>0$ (see Theorem \ref{Gconjec1}).

To state the results about the total mean curvature of NNSC fill-ins more concisely and systematically, we introduce the following notions. 
\begin{defn}\label{fillininvariant}
For an orientable closed null-cobordant Riemannian manifold $(\Sigma^{n},\gamma)$, define 
\begin{equation*}
\Lambda\left(\Sigma^{n},\gamma\right)=\sup\left\{\int_{\Sigma^n}H\, d\mu_\gamma\,\Big |\,(\Sigma^{n},\gamma,H)\text{ admits a NNSC fill-in}\right\},
\end{equation*}
and
\begin{equation*}
\Lambda_+\left(\Sigma^{n},\gamma\right)=\sup\left\{\int_{\Sigma^n}H\,d\mu_\gamma\,\Big |\,H>0,\,(\Sigma^{n},\gamma,H)\text{ admits a NNSC fill-in}\right\}.
\end{equation*}
More generally, for $\kappa\in\mathbf R$, define
\begin{equation*}
\Lambda_{+,\,\kappa}\left(\Sigma^{n}, \gamma\right)=\sup\left\{\int_{\Sigma^n} H d \mu_{\gamma}\,\bigg |
\begin{array}{c}
H>0,\  \,\left(\Sigma^{n}, \gamma, H\right) \text { admits a fill-in}\\
\text{with scalar curvature } R\geq n(n+1)\kappa
\end{array}\right\}.
\end{equation*}
\end{defn}
\begin{rem}
Theorem \ref{pscextension1} guarantees that the $\Lambda$-invariant is well-defined. For $\Lambda_+$, the supremum of an empty set is conventionally $-\infty$.
Note that $\Lambda_+$ was already defined for orientable closed surfaces in \cite{MM}.
\end{rem}
By definition, $\Lambda_{+,\,0}=\Lambda_{+}$. And it is obvious that $
\Lambda_{+,\,\kappa}\geq \Lambda_{+}$ for any $\kappa\leq 0$. We find that $\Lambda_{+,\,\kappa}$ has a deep relationship with the positive mass theorems (PMT) for asymptotically hyperbolic (AH) manifolds and asymptotically flat (AF) manifolds. Let $\left(\mathbf{S}^{n-1}, \gamma_{\rm std}\right)$ be the  $(n-1)$-sphere with the standard metric and denote its volume by $\omega_{n-1}$. We find the following relation

\begin{thm}\label{pmt1.1}
$\Lambda_{+}\left(\mathbf{S}^{n-1}, \gamma_{\rm std}\right)=(n-1)\omega_{n-1}$ is equivalent to PMT holds for AF $n$-manifolds. 
\end{thm}
This conclusion is also true for closed convex hypersurfaces in $\mathbf{R}^n$, see Theorem \ref{pmt2} below. For the notions of AF manifolds, ADM mass and the exact meaning of {\it ``PMT holds for AF manifolds''}, see Definition \ref{AFmanifold} and Definition \ref{ADMmass}. Combining Corollary 2.1 in \cite{MST}, Theorem 1.2 in \cite{WM} and Theorem \ref{pmt1.1}, we have the following interesting corollary:   	 
\begin{cor}
If $\Lambda_{+}\left(\mathbf{S}^{n-1}, \gamma_{\rm std}\right)$ can be achieved, then it is achieved only by $\mathbf D^n$ with the flat metric; $\Lambda_{+}\left(\mathbf{S}^{n-1}, \gamma_{\rm std}\right)$ can be achieved is equivalent to PMT holds for AF $n$-manifolds.	  	  
\end{cor}
   
Recently, Chru\'sciel-Delay \cite{CD} proved that the AF spacetime positive mass theorem implies the hyperbolic positive energy theorem. The AF spacetime positive mass theorem has been proved by Eichmair-Huang-Lee-Schoen \cite{EHLS} for dimensions not greater than $7$. By virtue of $\Lambda_{+,\,\kappa}$, we also build a bridge between the two PMTs. 

\begin{thm}\label{ahafpmt1}
PMT  holds for AH $n$-manifolds implies PMT holds for AF $n$-manifolds.
\end{thm}

For the definition of AH manifolds and the exact meaning of  {\it ``PMT holds for AH manifolds''}, see Definition \ref{AHdef}. 

A crucial ingredient of the proof of Theorem \ref{ahafpmt1} is the more delicate relation, i.e., {\it if for some $\lambda_0>0$,} $$\Lambda_{+,\,-1}\left(\mathbf{S}^{n-1}, \lambda_0^2\gamma_{\rm std}\right)=(n-1)\omega_{n-1}\sqrt{1+\lambda_0^2},$$
{\it then the same equality holds for all $\lambda\leq\lambda_0$. And this implies $$\Lambda_{+}\left(\mathbf{S}^{n-1}, \gamma_{\rm std}\right)=(n-1)\omega_{n-1},$$ which is equivalent to the PMT for AF $n$-manifolds. On the other hand, if PMT holds for AH manifolds, then 
$$\Lambda_{+,\,-1}\left(\mathbf{S}^{n-1}, \lambda^2\gamma_{\rm std}\right)=(n-1)\omega_{n-1}\sqrt{1+\lambda^2},$$ 
for all $\lambda>0$.}
We will investigate the reverse relation between the two PMTs elsewhere.

 With these things in mind, it might be important to find explicit estimates for $\Lambda_+$-invariant or $\Lambda$-invariant. Generally, it won't be easy. To us, a natural problem is  the following conjecture by Gromov (see p.31 in \cite{Gro19}).

\begin{conj}\label{conj1}
Let $(X,g)$ be a compact Riemannian manifold with boundary $Y$ and nonnegative scalar curvature. Suppose $Y$ is $\lambda$-bi-Lipschitz homeomorphic to the unit sphere $\mathbf{S}^{n-1}$. Let $H$ denote the mean curvature of $Y$ with respect to the outward normal. Then
\begin{equation*}
\int_{Y}H\,d\mu \leq C(\lambda)(n-1)\omega_{n-1},	
\end{equation*}
where $C(\lambda)$ is a constant depending only on $\lambda$ and $C(\lambda)\rightarrow 1$ as $\lambda\rightarrow 1$.
\end{conj}

We begin to explore the above conjecture for the surface case. Under the condition of nonnegative Gauss curvature and positive mean curvature, we can prove that $\Lambda_+(\mathbf{S}^2, \gamma)$  is close to $\Lambda_+(\mathbf{S}^2, \gamma_0)$ provided $(\mathbf{S}^2, \gamma)$  is close to $(\mathbf{S}^2, \gamma_0)$ in the Gromov-Hausdorff topology. For explicit statement, see Theorem \ref{2dstability} and Remark \ref{GHrem}.

In \cite{MM}, $\Lambda_+(\Sigma^2,\gamma)$ was used to define Brown-York mass for closed surfaces. In this paper, we generalize it to higher dimensions. For a triple of Bartnik data $(\Sigma^{n-1},\gamma,H)$, we define its  generalized Brown-York mass  by
$$m_{\rm BY}\left(\Sigma^{n-1},\gamma,H\right)=\frac{1}{(n-1)\omega_{n-1}}\left(\Lambda_+\left(\Sigma^{n-1},\gamma\right)-\int_{\Sigma^{n-1}}H\, d\mu_\gamma\right).$$
When considering an oriented compact Riemannian manifold $(\Omega^n,g)$ with boundary, we define the generalized Brown-York mass by
$$m_{\rm BY}\left(\partial \Omega^n,g\right)=\frac{1}{(n-1)\omega_{n-1}}\left(\Lambda_+\left(\partial\Omega^n,g|_{\partial\Omega}\right)-\int_{\partial\Omega}H_g\, d\mu_{g|_{\partial\Omega}}\right).$$

A key point to prove Theorem \ref{pmt1.1} is to show the limit of the generalized Brown-York mass of the large coordinate spheres is the ADM mass of the ambient manifold. In surface case, i.e., $n=3$, that can be achieved by arguments from isometric embedding theory (see \cite{FST}). But in our current case, such techniques fail; we turn to use monotonicity lemma (Lemma \ref{c0monotonicity}) to overcome this difficulty. 

By analyzing the asymptotic behavior of the generalized Brown-York mass at the infinity of asymptotically Schwarzschild (AS) manifolds, we obtain another equivalent statement of the PMT for AF $n$-manifolds: there exists an AS $n$-manifold of which the large sphere limit of the generalized Brown-York mass is finite. For details, see Remark \ref{gbylimitas}.  

From the results in \cite{ST02} and Theorem \ref{c111}, we see that $m_{\rm BY}(\partial \Omega^3,g)$ is just the same as the original definition in \cite{BrownYork-1, BrownYork-2} when the Gauss curvature of $\partial \Omega^3$ is quasi-positive (nonnegative and positive somewhere). As a corollary of Theorem \ref{2dstability}, we have:

\begin{cor}\label{cby}
Let $\Omega$ be a compact $3$-manifold with boundary and $\{g_i\}$ a sequence of smooth metrics on $\Omega$. Assume the Gauss curvature of $\partial\Omega$ with respect to $g_i$ is quasi-positive. If $g_i$ converges to a smooth metric $g$ in the $C^1$-topology, then $m_{\rm BY}(\partial \Omega, g_i)$ converges to $m_{\rm BY}(\partial \Omega, g)$.
\end{cor}

Note that the mean curvature $H_0$ of $\partial\Omega$ when isometric embedded in $\mathbf{R}^3$ is involved in the original definition of  Brown-York mass (\cite{BrownYork-1, BrownYork-2}). Therefore, in order to estimate the Brown-York mass, traditionally one has to estimate $H_0$, so high order derivatives convergence of $g_i$ is needed in this context. In above corollary, by using the monotonicity property of $\Lambda_+(\mathbf{S}^2,\gamma)$, we only need $C^1$-convergence of metrics. 

We  have similar results for higher dimensional cases. To state them concisely, we need some notations. Let $\mathcal M(\Sigma)$ be the space of all smooth Riemannian metrics on $\Sigma$, and 
$\mathcal M_{\rm psc}\left(\Sigma\right)=\left\{\gamma\in\mathcal M\left(\Sigma\right)\,|\,R_{\gamma}>0\right\}$. For $p>n$, $G>0$, define
$\mathcal M^n_{p,\,G}=\{\gamma\in\mathcal M_{\rm psc}(\mathbf S^n)\,|\,\|\nabla_{\gamma_{\rm std}}\gamma\|_{L^p(\mathbf S^n,\gamma_{\rm std})}\leq G\}$.

\begin{thm}\label{continuity1-re0}
Let $2\leq n\leq 6$. Given constants $p>n$ and $G>0$, for any $\varepsilon>0$, there is a $\delta(n,p,G,\varepsilon)>0$ such that for any $\gamma\in\mathcal M^n_{p,\,G}$ with $\|\gamma-\gamma_{\rm std}\|_{L^\infty(\mathbf{S}^{n},\gamma_{\rm std})}\leq\delta$, there holds 
\begin{equation*}
\left |\Lambda_+(\mathbf{S}^{n},\gamma)-\Lambda_+(\mathbf{S}^{n},\gamma_{\rm std})\right |\leq\varepsilon.
\end{equation*}	
\end{thm}

One crucial step to prove Theorem \ref{continuity1-re0} is making use of the $C^0$-local connectedness of PSC metrics (Proposition \ref{localconnectednesspsc}) and quasi-spherical metrics. Note that uniform $W^{1,\,p}$-boundedness is satisfied by families of metrics possess certain compactness. As an application of Theorem \ref{continuity1-re0}, we have

 \begin{thm}\label{main0}
Let $2\leq n\leq 6$. For any $\varepsilon>0$, $K>0$, and $i_0>0$,  there exists a constant 
$\delta=C(n,K, i_0,\varepsilon)>1$
such that for any $\gamma\in\mathcal N_n\left(K, i_0\right)$ with $\dial(\gamma)\leq\delta$, there holds
\begin{equation*}
\left |\Lambda_+(\mathbf{S}^{n},\gamma)-\Lambda_+(\mathbf{S}^{n},\gamma_{\rm std})\right |\leq\varepsilon.
\end{equation*}	
Notations $\dial(\gamma)$ and $\mathcal N_n\left(K, i_0\right)$ are given on p.24 and p.35 below respectively. 	
\end{thm}
\begin{rem}
The assumption $2\leq n\leq 6$ in Theorem \ref{finiteinfmeancurpsc1}, Theorem \ref{continuity1-re0} and Theorem \ref{main0} is only due to the dimension restriction of the classical AF positive mass theorem by Schoen-Yau \cite{SY79a,SY79b,SY81} (see also \cite{Witten} for the spin case). Recently, Schoen-Yau \cite{SY17} claimed that the dimension restriction can be dropped; thus, these results hold in all dimensions not less than $2$.
\end{rem}

This paper is organized as follows. In Section \ref{section2}, we prove the extensibility of an arbitrary boundary metric to a PSC metric inside, and show pointwisely large mean curvature is an obstruction for such extensions. In Section \ref{section3}, we study the the relationship between $\Lambda_{+,\,\kappa}$ and PMTs for AF and AH manifolds. In Section \ref{section4}, we prove the finiteness of the total mean curvature for NNSC spin fill-ins of spheres under the positive mean curvature condition. In Section \ref{section5}, we give some estimates for $\Lambda_+(\mathbf{S}^2, \gamma)$ and prove Conjecture \ref{conj1} for nonnegatively curved surfaces with the positive mean curvature. In Section \ref{section6}, we show $C^0$-local path-connectedness for PSC metrics on a fixed closed manifold; based on this, we give some stability estimates for $\Lambda_+(\mathbf{S}^n, \gamma)$ when $n\geq3$. 
 
 \medskip
 {\it Acknowledgements.} We would like to express our sincere gratitude to Misha Gromov for his interest, comments and suggestions on this work. We would like to thank Luen-Fai Tam  and Pengzi Miao for their interest in this work, especially, for Luen-Fai Tam for informing us Theorem 1.4 in \cite{KW} and suggesting us to clarify the constant in Theorem \ref{finiteinfmeancurpsc1}. We are also very grateful to Roman Prosanov for the discussion about Pogorelov's rigidity theorem \cite{P}.

%--------------------------------------------------------------------------------------------------------------------------------------------------------------------------------------------------------------------------------------------------
%--------------------------------------------------------------------------------------------------------------------------------------------------------------------------------------------------------------------------------------------------

\section{Extensibility of boundary metrics to PSC metrics inside and obstructions on the mean curvature}\label{section2}

In this section, we prove Theorem \ref{pscextension1} and Theorem \ref{finiteinfmeancurpsc1}. To this end, we first exploit the PSC-cobordism of Bartnik data. Recall that a triple of Bartnik data $(\Sigma,\gamma,H)$ consists of an oriented closed null-cobordant Riemannian manifold $(\Sigma,\gamma)$ and a given smooth function $H$ on $\Sigma$.
\begin{defn}[\cite{SH}]
Given Bartnik data $(\Sigma^{n}_i,\gamma_i,H_i)$ $(i=0,1)$, we say $(\Sigma^{n}_0,\gamma_0,H_0)$ is PSC-cobordant to $(\Sigma^{n}_1,\gamma_1,H_1)$, if there is an orientable $(n+1)$-dimensional manifold $(\Omega^{n+1},g)$ such that: \\
$(1)$ $\partial\Omega^{n+1}=\Sigma^{n}_0\sqcup\Sigma^{n}_1;$ \\
$(2)$ the induced metric on $\Sigma^{n}_i$ is $\gamma_i$, and the mean curvature of $\Sigma^n_i$ with respect to the outward normal is $H_i$ for $i=0,1;$ \\
$(3)$ $R_g>0$.
\end{defn}
\begin{lm}\label{PSCcobordismlm}
Let $\Sigma$ be a closed manifold, and $\gamma_0$, $\gamma_1$ be two smooth metrics on $\Sigma$ with $\gamma_1>\gamma_0$. Then for any smooth function $h$ on $\Sigma$, there are smooth functions $h_0$ and $h_1$ on $\Sigma$, such that $(\Sigma,\gamma_0,h_0)$ is PSC-cobordant to $(\Sigma,\gamma_1,h_1)$ and $h_1>h$. Moreover, the corresponding region is diffeomorphic to $\Sigma\times[0,1]$.   
\end{lm}
We will construct the required PSC-cobordism via a quasi-spherical type equation. To begin with, we state some important properties for quasi-spherical type equations in the following two claims. 

\begin{cla}\label{quassphericalscalar}
Let $\Sigma$ be a closed manifold and $\{\gamma_t\}_{t\in [0,1]}$ be a smooth path of metrics on $\Sigma$. Set $\Omega=\Sigma\times[0,1]$. On $\Omega$, define the metric 
$\bar g=dt^2+\gamma_t$. For a smooth positive function $u$ on $\Omega$, let $g=u^2dt^2+\gamma_t$. Then the scalar curvature of $g$ is 
\begin{equation*}
R_g=u^{-2}R_{\bar g}+\left(1-u^{-2}\right)R_{\gamma_t}+2u^{-3}u_t\bar H_t-2u^{-1}\Delta_{\gamma_t}u,
\end{equation*}
where $\bar H_t$ is the mean curvature of the slice $\Sigma\times\{t\}$ with respect to the $\partial_t$-direction in $(\Omega,\bar g)$.
In particular, if $u$ satisfies the following quasi-spherical equation
\begin{equation}\label{quassphericalscalareq0}
\bar H_t\frac{\partial u}{\partial t}=u^2\Delta_{\gamma_t}u+\frac{1}{2}\left(f-R_{\gamma_t}\right)u^3+\frac{1}{2}\left(R_{\gamma_t}-R_{\bar g}\right)u
\end{equation}
for some smooth function $f$ on $\Omega$, then $R_g\equiv f$. 
\end{cla}

\begin{cla}\label{quassphericalestimate}
Let $\Sigma$ be a closed manifold and $\{\gamma_t\}_{t\in [0,1]}$ be a smooth path of metrics on $\Sigma$. Set $\Omega=\Sigma\times[0,1]$. On $\Omega$, define the metric 
$\bar g=dt^2+\gamma_t$. Let $\bar H_t$ denote the mean curvature of the slice $\Sigma\times\{t\}$ with respect to the $\partial_t$-direction in $(\Omega,\bar g)$. Assume $\bar H_t>0$. Then for any smooth function $f$ on $\Omega$, there is a constant $\varepsilon_0>0$ depending only on $M$, where 
$M=\max_{\Omega}\bar H^{-1}_t\left(|R_{\gamma_t}|+|R_{\bar g}|+|f|\right)$, such that for any $\varepsilon\in(0,\varepsilon_0]$, the following quasi-spherical metric equation 
\begin{equation}\label{PSCExtensionQSeq}
\left\{
\begin{aligned}
\bar H_t\frac{\partial u}{\partial t}&=u^2\Delta_{\gamma_t}u+\frac{1}{2}\left(f-R_{\gamma_t}\right)u^3+\frac{1}{2}\left(R_{\gamma_t}-R_{\bar g}\right)u,\\
\quad u(\cdot,0)&\equiv\varepsilon
\end{aligned}
\right.
\end{equation}
defined on $\Omega$ has a positive solution and the solution satisfies
\begin{equation}\label{quassphericaltwobounds}
\frac{1}{2}e^{-M}\varepsilon\leq u\leq e^{M}\varepsilon.
\end{equation}
 
\end{cla}

\begin{proof}[Proof of Claim \ref{quassphericalscalar}]
For convenience, denote $\Sigma\times\{t\}$ by $\Sigma_t$. 
Let $\bar A_t$ be the second fundamental form of $\Sigma_t$ induced from $\bar g$ with respect to the $\partial_t$-direction. Let $A_t$ and $H_t$ be the second fundamental form and the mean curvature of $\Sigma_t$ induced from $g$ with respect to the $\partial_t$-direction. It is not hard to see 
\begin{equation}\label{sec2meancurrelation}
A_t=u^{-1}\bar A_t,\quad H_t=u^{-1}\bar H_t.
\end{equation}
By the Gauss equation, we have
\begin{equation}\label{Gausseq}
R_g=2\ric_g(\nu,\nu)+R_{\gamma_t}-\left(H_t^2-|A_t|^2\right),
\end{equation}
where $\nu=u^{-1}\partial_t$.
By the ``pointwise" second variation formula, 
\begin{equation}\label{Jacobieq}
\partial_t H_t=-\Delta_{\gamma_t}u-\left(|A_t|^2+\ric_g(\nu,\nu)\right)u.
\end{equation}
Combining \eqref{Gausseq} and \eqref{Jacobieq}, we get
\begin{equation}\label{relation1}
R_g=R_{\gamma_t}-\left(H_t^2+|A_t|^2\right)-2u^{-1}\Delta_{\gamma_t}u-2u^{-1}\partial_t H_t.
\end{equation}
Letting $u\equiv 1$ in \eqref{relation1}, for $\bar g$ we have 
\begin{equation}\label{relation2}
R_{\bar g}=R_{\gamma_t}-\left(\bar H_t^2+|\bar A_t|^2\right)-2\partial_t \bar H_t.
\end{equation}
Substituting \eqref{relation2} into \eqref{relation1} and using the relation \eqref{sec2meancurrelation}, we get
\begin{equation*}
R_g=u^{-2}R_{\bar g}+\left(1-u^{-2}\right)R_{\gamma_t}-2u^{-1}\Delta_{\gamma_t}u+2\bar H_tu^{-3}u_t.
\end{equation*}
Finally, due to \eqref{quassphericalscalareq0}, we obtain $R_g\equiv f$.

For more detailed calculations for the expression of the scalar curvature by moving frames method, please see p.82--84 in \cite{ST02}. 
\end{proof}

\begin{proof}[Proof of Claim \ref{quassphericalestimate}]
By the theory of parabolic equations, equation \eqref{PSCExtensionQSeq} has a local solution (see \cite{Bartnik1}). Assume $[0,T)$ is the largest interval where the solution exists. We first prove that $u$ has the estimate \eqref{quassphericaltwobounds} in $[0,T)$ for sufficiently small $\varepsilon$. 

Consider the following corresponding ODEs: 
\begin{equation}\label{supersolution}
\left\{
\begin{aligned}
\frac{\partial v}{\partial t}&=-\frac{M}{2}(v^3+v),\\
\quad v(0)&\equiv\varepsilon,
\end{aligned}
\right.
\end{equation}
and
 \begin{equation}\label{subsolution}
\left\{
\begin{aligned}
\frac{\partial w}{\partial t}&=\frac{M}{2}(w^3+w),\\
\quad w(0)&\equiv\varepsilon.
\end{aligned}
\right.
\end{equation}
Equation \eqref{supersolution} has a positive solution on the whole $[0,1]$, namely
\begin{equation*}
v(t)=\frac{\varepsilon}{\sqrt{(1+\varepsilon^2)e^{Mt}-\varepsilon^2}}.
\end{equation*}
Equation \eqref{subsolution} has a positive solution on $\big[0,\frac{1}{M}\log (1+\varepsilon^{-2})\big)$, namely
\begin{equation*}
w(t)=\frac{\varepsilon}{\sqrt{(1+\varepsilon^2)e^{-Mt}-\varepsilon^2}}.
\end{equation*}
Set $\varepsilon_0=e^{-\frac{M}{2}}$ and choose an $\varepsilon\leq\varepsilon_0$. Then $\log (1+\varepsilon^{-2})\geq 1$, and in $[0,1]$ there holds
$$(1+\varepsilon^2)e^{-Mt}-\varepsilon^2\geq e^{-2M}.$$ Thus $w$ exists on the whole $[0,1]$, and satisfies
\begin{equation*}
w(t)\leq e^M\varepsilon.
\end{equation*}
Next, we prove $v\leq u\leq w$ in $[0,T)$. Take $\Phi=w-u$. Then $\Phi(x,0)\equiv 0$ and $\Phi$ satisfies
\begin{align*}
\frac{\partial\Phi}{\partial t}=&\frac{u^2}{\bar H_t}\Delta_{\gamma_t}\Phi+\frac{1}{2\bar H_t}\left[\left(f-R_{\gamma_t}\right)\left(u^2+uw+w^2\right)+\left(R_{\gamma_t}-R_{\bar g}\right)\right]\Phi\\
&+\frac{1}{2}\left(M-\frac{f-R_{\gamma_t}}{\bar H_t}\right)w^3+\frac{1}{2}\left(M-\frac{R_{\gamma_t}-R_{\bar g}}{\bar H_t}\right)w\\
\geq &\frac{u^2}{\bar H_t}\Delta_{\gamma_t}\Phi+\frac{1}{2\bar H_t}\left[\left(f-R_{\gamma_t}\right)\left(u^2+uw+w^2\right)+\left(R_{\gamma_t}-R_{\bar g}\right)\right]\Phi.
\end{align*}
By the parabolic maximum principle, $\Phi\geq 0$ in $[0,T)$, namely $u\leq w$ in $[0,T)$. Similarly, we can prove $u\geq v$ in $[0,T)$.  
Once we have the estimate \eqref{quassphericaltwobounds} in $[0,T)$, by the regularity theory of parabolic equations, $u$ can be smoothly extended to $T$. Then $u$ can also be extended across $T$ a bit if $T<1$ by local existence. This completes the proof. 
\end{proof}

\begin{proof}[Proof of Lemma \ref{PSCcobordismlm}]
Set $\gamma_t=(1-t)\gamma_0+t\gamma_1$. Then $\{\gamma_t\}_{t\in [0,1]}$ is a smooth path of metrics on $\Sigma$. Set $\Omega=\Sigma\times[0,1]$. On $\Omega$, define $\bar g=dt^2+\gamma_t$. Denote the slice $\Sigma\times\{t\}$ by $\Sigma_t$. Let $\bar A_t$ and $\bar H_t$ be the second fundamental form and the mean curvature of $\Sigma_t$ induced from $\bar g$ with respect to the $\partial_t$-direction. It is not hard to see
$$\bar A_t=\frac{1}{2}\gamma'_t=\frac{1}{2}\left(\gamma_1-\gamma_0\right).$$
Since $\gamma_1>\gamma_0$, $\bar H_t>0$. 
For constants $\delta>0$ and $\varepsilon>0$, consider the following quasi-spherical metric equation on $\Sigma\times[0,1]$: 
\begin{equation*}
\left\{
\begin{aligned}
\bar H_t\frac{\partial u}{\partial t}&=u^2\Delta_{\gamma_t}u+\frac{1}{2}\left(\delta-R_{\gamma_t}\right)u^3+\frac{1}{2}\left(R_{\gamma_t}-R_{\bar g}\right)u,\\
\quad u(\cdot,0)&\equiv\varepsilon.
\end{aligned}
\right.
\end{equation*}
By Claim \ref{quassphericalestimate}, for sufficiently small $\varepsilon$, above equation has a positive solution on $\Omega$ and the solution satisfies $u\leq e^M\varepsilon$,
where $M$ is a constant independent of $\varepsilon$. 

Set 
$
g=u^2dt^2+\gamma_t.
$
By Claim \ref{quassphericalscalar}, we have $R_g\equiv\delta>0$. 
Let $A_t$ and $H_t$ denote the second fundamental form and the mean curvature of $\Sigma_t$ induced from metric $g$ with respect to the $\partial_t$-direction. Since
$H_t=u^{-1}\bar H_t$, by the estimate \eqref{quassphericaltwobounds} of $u$, we get
\begin{equation*}
H_1\geq \varepsilon^{-1}e^{-M}\min_{\Sigma_1}\bar H_1.
\end{equation*}
Hence, we can choose a sufficiently small $\varepsilon$ so that $H_1>h$. Set $h_0=-H_0$, $h_1=H_1$. Note that on $\Sigma_1$, the outward direction is the same as the $\partial_t$-direction; but on $\Sigma_0$, the outward direction is opposite to the $\partial_t$-direction. Having above analysis, we conclude that $(\Sigma,\gamma_0,h_0)$ is PSC-cobordant to $(\Sigma,\gamma_1,h_1)$ with $h_1>h$.
\end{proof}
\medskip
\begin{center}
\begin{tikzpicture}
\draw[semithick] (0,2) ellipse (2 and 0.7);
\draw[semithick] (-1,-2) arc(180:360:1 and 0.3);
\draw[dashed, color=gray] (-1,-2) arc(-180:-360:1 and 0.3);
\draw[semithick] (-1.3,-0.5) arc(180:360:1.3 and 0.4);
\draw[dashed, color=gray] (-1.3,-0.5) arc(-180:-360:1.3 and 0.4);
\draw plot[smooth] coordinates {(-1,-2) (-1.3,-0.5) (-2,2)};
\draw plot[smooth] coordinates {(1,-2) (1.3,-0.5) (2,2)};
\draw (0,0.5) node {$(\Omega,g)$};
\draw[->] (2.1,2)--(2.5,2);
\draw (4.9,2) node {$(\Sigma_1,\gamma_1,h_1),\ \ h_1\triangleq H_1>h$};
\draw[->] (1.5,-0.5)--(1.9,-0.5);
\draw (2.6,-0.5) node {$(\Sigma_t,\gamma_t)$};
\draw[->] (1.1,-2)--(1.5,-2);
\draw (3.7,-2) node {$(\Sigma_0,\gamma_0,h_0),\ \ h_0\triangleq -H_0$};
\draw[->] (-1,2.64)--(-1.2,3.2);
%\draw (-1.4,2.9) node {$\nu$};
\draw (0.8,3.7) node {$\nu$: outward normal};
\draw[->] (-1,-2.15)--(-1.1,-2.8);
\draw (0.8,-3.2) node {$\nu$: outward normal};
\draw[->] (-3.4,-1.7)--(-3.4,1.9);
\draw (-3.8,0) node{t};
\end{tikzpicture}\\
\vspace{2ex}
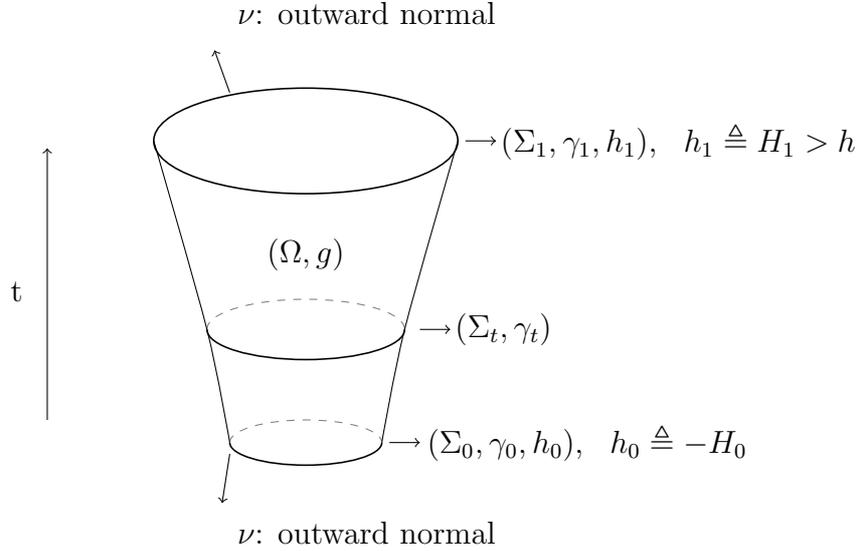
{\bf Fig 1.} PSC-cobordism of Bartnik data
\end{center}
\bigskip

\begin{proof}[Proof of Theorem \ref{pscextension1}]
By Theorem 1.4 in \cite{KW}, or more strongly, by Theorem 4.5.1 in \cite{Gro69}, there is a metric $g_1$ on $X$ with positive scalar curvature. We first assume $Y$ has only one connected component. Denote the induced metric on $Y$ from $g_1$ by $\gamma_1$, and the mean curvature of $Y$ in $(X,g_1)$ with respect to the outward normal by $h$. We may assume $\gamma_1>\gamma$ after a suitable rescaling. By Lemma \ref{PSCcobordismlm}, there are smooth functions $h_0$ and $h_1$ on $Y$, such that $(Y,\gamma,h_0)$ is PSC-cobordant to $(Y,\gamma_1,h_1)$ and $h_1>-h$. Denote the corresponding region of PSC-cobordism by $(\Omega,g_0)$. Glue $(\Omega,g_0)$ to $(X,g_1)$ along $(Y,\gamma_1)$. Denote the obtained manifold by $(\tilde X,\tilde g)$. It is not hard to see that $
\tilde g$ is Lipschitz across the corners $(Y,\gamma_1)$. Away from $(Y,\gamma_1)$, $\tilde g$ is smooth and $R_{\tilde g}>0$. Since $h_1>-h$ on the corners, after carrying out Miao's mollifiying procedure (\cite{Miao02}) for the corners along $(Y,\gamma_1)$ and a suitable conformal deformation that preserves the boundary metric, we may get a smooth PSC metric $g$ on $\tilde X$ with $g=\tilde g$ on $\partial X$. Apparently, $g|_{\partial \tilde X}=\gamma$. Since $\Omega$ is diffeomorphic to $Y\times[0,1]$, $\tilde X$ is diffeomorphic to $X$. Thus $g$ can be regarded as a metric on $X$. 

 When $Y$ has multiple connected components, the proof is essentially the same as above, except that we construct PSC-cobordism for every component. This completes the proof.
\end{proof}
\begin{rem}
From the proof of Theorem \ref{pscextension1} and Theorem 6 in \cite{SWWZ}, we see that there is a constant $C$ depending on the topology of $X$ and $\gamma$ such that we can prescribe any mean curvature $H\leq C$ for a PSC extension of $\gamma$. 
\end{rem}

But the following example shows that we can not even further prescribe zero mean curvature. 
\begin{exm}
The standard metric $\gamma_{\rm std}$ on $\mathbf S^2$ can not be extended to a NNSC metric on $\mathbf T^3\setminus\mathbf D^3$ with zero mean curvature. Otherwise, after capping $\mathbf T^3\setminus\mathbf D^3$ with the standard hemisphere, mollifying the corners and taking a suitable conformal deformation, we can get a smooth PSC metric on $\mathbf T^3$, which is a contradiction. 
\end{exm}

Now we give the proof of Theorem \ref{finiteinfmeancurpsc1}, which can be regarded as an application of the PSC cobordism Lemma \ref{PSCcobordismlm}.

\begin{proof}[Proof of Theorem \ref{finiteinfmeancurpsc1}]
By assumption, $\Sigma$ can be smoothly embedded in $\mathbf R^{n+1}$. Choose an arbitrary embedding map $F:\Sigma\rightarrow\mathbf R^{n+1}$. Then $F^*(g_{E})$ is a smooth metric on $\Sigma$, where $g_E$ denotes the standard Euclidean metric on $\mathbf R^{n+1}$. For $\lambda>0$, set $F_{\lambda}=\lambda\cdot F$. Obviously, $F^*_{\lambda}(g_E)=\lambda^2F^*(g_{E})$. So there is a $\lambda_0$ such that $F^*_{\lambda_0}(g_E)>\gamma$. 
Denote the mean curvature of $F_{\lambda_0}(\Sigma)$ in $\mathbf R^{n+1}$ with respect to the outward normal by $\tilde H$ and the exterior region of $F_{\lambda_0}(\Sigma)$ in $\mathbf{R}^{n+1}$ by $\tilde\Omega$. By Lemma \ref{PSCcobordismlm}, there are smooth functions $H_0$ and $H_1$ on $\Sigma$, such that $(\Sigma,\gamma,H_0)$ is PSC-cobordant to $(\Sigma,F^*_{\lambda}(g_E),H_1)$ via a Riemannian region $(\Sigma\times [0,1],\check g)$, and $H_1>\tilde H$. Take $C=\max_{\Sigma} (-H_0)$. 

Suppose $H$ is a smooth function on $\Sigma$ that satisfies $H>C$, and $(\Omega, g)$ is a fill-in of $(\Sigma,\gamma,H)$ with NNSC. Glue $(\Omega, g)$ to $(\Sigma\times [0,1],\check g)$ along $(\Sigma\times\{0\},\gamma)$, then glue the obtained manifold to $(\tilde\Omega,g_E)$ along $(\Sigma\times\{1\},F^*_{\lambda_0}(g_E))$. Denote the new manifold by $(M,\hat g)$. Then $(M,\hat g)$ is a complete manifold with corners and a flat end. The scalar curvature of $(M,\hat g)$ is nonnegative away from the corners. Along the corners $\Sigma\times\{0\}$ and $\Sigma\times\{1\}$, the mean curvature with respect to the inside metric in the outgoing direction is strictly greater than the mean curvature with respect to the outside metric in the outgoing direction, where the outgoing direction is pointing towards infinity. By Miao's result \cite{Miao02}, $(M,\hat g)$ must has positive mass, which contradicts the flatness of the end. Consequently, for pointwisely sufficient large $H$, $(\Sigma,\gamma,H)$ admits no fill-in of NNSC. 

From the proof, we see the threshold $C$ depends only on $(\Sigma,\gamma)$ and the embedding $F$. Finally we take the infimum of $C$ for all embeddings to get a constant depending only on $(\Sigma,\gamma)$ but not the embedding. This completes the proof. 
\end{proof}
Theorem \ref{finiteinfmeancurpsc1} can be generalized to fill-ins with scalar curvature bounded from below. 
\begin{thm}
Suppose $\Sigma\in\mathcal C_{n}$ $(2\leq n\leq 6)$. For any metric $\gamma$ on $\Sigma$ and any constant $\kappa$, there is a constant $h_0(\Sigma,\gamma,\kappa)$ such that for any smooth function $H$ on $\Sigma$ satisfying $H>h_0(\Sigma,\gamma,\kappa)$, $(\Sigma,\gamma,H)$ admits no fill-in of metrics with scalar curvature $R\geq n(n+1)\kappa$.
\end{thm}

\begin{proof}[Sketch of Proof.]
The proof is essential the same as the proof of Theorem \ref{finiteinfmeancurpsc1}. The difference is that we embed $(\Sigma,\gamma)$ into an anti-de Sitter-Schwarzschild end with negative mass, namely
a rotational symmetric AH end with scalar curvature $R=-n(n+1)$ and mass $m<0$. The required embedding exists since $\Sigma$ can be embedded in $\mathbf R^{n+1}$. To draw the contradiction, we first take a suitable scaling, then carry out the mollifying procedure for the corners in \cite{Miao02} and the conformal deformation in \cite{BQ}, finally use the positive mass theorem for AH manifolds with negative mass aspect function (Theorem 3 in \cite{ACG}). 
\end{proof}

 %%----------------------------------------------------------------------------------------------------------------------------------------------------------------------------------------------------------------------------------------------------------------------------------
 %%----------------------------------------------------------------------------------------------------------------------------------------------------------------------------------------------------------------------------------------------------------------------------------

\section{Relationship between $\Lambda_{+,\,\kappa}$ and positive mass theorems}\label{section3}

In this section, we investigate the relationship between the quantity $\Lambda_{+,\,\kappa}\left(\mathbf{S}^{n-1}, \gamma_{\rm std}\right)$ and positive mass theorems for AF and AH manifolds. Let us begin with some important notions. 

\begin{defn}\label{AFmanifold}
Let $n\geq 3$. A Riemannian manifold $(M^n, g)$ is said to be asymptotically flat (AF) if there is a compact set $K\subset M^n$ such that $M^n\setminus K$ is diffeomorphic to $\mathbf{R}^n\setminus B_1(0)$, and in this coordinates, $g$ satisfies
$$
\left|g_{i j}-\delta_{i j}\right|+|x|\left|\partial g_{i j}\right|+|x|^{2}\left|\partial^{2} g_{i j}\right|+|x|^{3}\left|\partial^{3} g_{i j}\right|
=O\left(|x|^{-p}\right)$$
for some $p>\frac{n-2}2$. Furthermore, we require that
$$\int_{M^n}|R_g|\,d\mu_g<\infty.$$
\end{defn}
\begin{defn}
An AF manifold $(M^n, g)$   is called asymptotically Schwarzschild (AS) if there is a compact set $K\subset M^n$ such that $M^n\setminus K$ is diffeomorphic to $\mathbf{R}^n\setminus B_1(0)$, and in this coordinates, $g$ satisfies
$$g_{ij}=\left(1+\frac{2m}{n-2}|x|^{2-n}\right)\delta_{ij}+b_{ij},$$
where $m$ is a constant and $b_{ij}$ decays as
$$\left|b_{ij}\right|+|x|\left|\partial b_{ij}\right|+|x|^{2}\left|\partial^{2} b_{ij}\right|+|x|^{3}\left|\partial^{3} b_{ij}\right|
=O\left(|x|^{1-n}\right).$$
\end{defn}
For AF manifolds, we can define a conserved quantity called the ADM mass.

\begin{defn}[\cite{ADM}]\label{ADMmass}
The  Arnowitt-Deser-Misner (ADM) mass  of an AF manifold $(M^n,g)$ is defined by
\begin{equation*}
m_{\rm ADM}(M^n,g)=\lim_{r\to\infty}\frac{1}{{ 2(n-1)}\omega_{n-1}}\int_{S_r}
\left(g_{ij,i}-g_{ii,j}\right)\nu^j\,dS_r,
\end{equation*}
where $S_{r}$ is the coordinate sphere near the infinity, $\nu$ is the Euclidean unit outward normal to $S_r$, and $dS_r$ is the Euclidean area element on $S_r$.
\end{defn}
We say that {\it positive mass theorem holds for AF $n$-manifolds if the ADM mass of any AF $n$-manifold with nonnegative scalar curvature is nonnegative, and the ADM mass vanishes if and only if this AF manifold is isometric to $\mathbf{R}^n$.} 

The last notion is asymptotically hyperbolic manifolds. Here we follow the definition by X.-D. Wang \cite{Wang}.
\begin{defn}[\cite{Wang}]\label{AHdef}
A complete noncompact Riemannian manifold $(M^n,g)$ is said to be asymptotically hyperbolic (AH) if there is a compact manifold $X$ with boundary and a smooth
function $\rho$ on $ X$ such that: 
\begin{enumerate}
\item[(i)] $M^n$ is differeomorphic to  $X\setminus \partial X$ (we identify $M^n$ with $X\setminus \partial X$ in the sequel);
\item [(ii)]$\rho=0$ on $\partial X$, and $\rho>0$ on $X\setminus \partial X$;
\item[(iii)] $\bar g=\rho^2g$ extends to a smooth metric on $X$;
\item[(iv)] $|d\rho |_{\bar g}=1$ on $\partial X$;
\item [(v)] each component $\Sigma$ of $\partial X$ is the standard round sphere $(\mathbf{S}^{n-1},\gamma_{\rm std})$, and in a collar neighborhood of $\Sigma$, 
$$g=\sinh^{-2}\rho\left(d\rho^2+\gamma_\rho\right)$$ with
$\gamma_\rho$ being a $\rho$-dependent family of metrics on $\mathbf{S}^{n-1}$ that has the expansion
   $$
   \gamma_\rho=\gamma_{\rm std}+\frac{\rho^n}{n}h+O\left(\rho^{n+1}\right),
   $$
where $h$ is a $C^{n-1}$  symmetric $2$-tensor on $\mathbf{S}^{n-1}$.
\end{enumerate}
\end{defn}

Let  $(M^n,g)$ be an AH manifold. For simplicity, we assume $M^n$ has only one end. We say {\it PMT holds for AH manifolds, if for any such $(M^n,g)$, the condition $R_g\geq -n(n-1)$ implies
\begin{equation*}
\int_{\mathbf{S}^{n-1}}\tr_{\gamma_{\rm std}}h\,d\mu_{\gamma_{\rm std}}\geq 0,
\end{equation*}
and the equality holds only when $(M^n,g)$ is isometric to the hyperbolic space $\mathbf{H}^n$.} For some results on PMT for AH manifolds, please see \cite{Wang,CH,ACG,CD}.

The first result in this  section is the following.

\begin{thm}\label{pmt1}
$\Lambda_{+}\left(\mathbf{S}^{n-1}, \gamma_{\rm std}\right)=(n-1)\omega_{n-1}$ is equivalent to PMT holds for AF $n$-manifolds. 
 \end{thm}
 
 Before prove Theorem \ref{pmt1}, we introduce the following useful lemma, which is on the monotonicity of $\Lambda_+$-invariant. 
 \begin{lm}\label{c0monotonicity}
Let $\Sigma$ be a closed Riemanniann manifold and $\{\gamma_t\}_{t\in[0,1]}$ be a smooth path of NNSC metrics on $\Sigma$. Assume $\gamma_t$ monotonically increases, namely $\gamma_{t_2}\geq\gamma_{t_1}$ for $t_2\geq t_1$. Then
$$\Lambda_+\left(\Sigma,\gamma_0\right)\leq\Lambda_+\left(\Sigma,\gamma_1\right).$$
\end{lm}
\begin{proof}
For an arbitrary $\varepsilon>0$, let $\bar\gamma_t=e^{2\varepsilon t}\gamma_t$, and let $\bar g=dt^2+\bar\gamma_t$ on $\Sigma\times [0,1]$. Denote $\Sigma\times\{t\}$ by $\Sigma_t$. Let $\bar A_t$ and $\bar H_t$ be the second fundamental form and the mean curvature of $\Sigma_t$ induced from metric $\bar g$ respectively. Clearly, $\bar\gamma_t$ strictly monotonically increases, so $\bar A_t>0$. It follows that $\bar H_t>0$ and $\bar H_t^2-\|\bar A_t\|^2>0$. Assume $(\Omega,\tilde g)$ is a NNSC fill-in of $(\Sigma,\gamma_0,H)$ for some smooth function $H>0$ on $\Sigma$. 

Consider the quasi-spherical metric equation
\begin{equation}\label{Eq: quasi spherical equation 2}
\left\{
\begin{aligned}
\bar H_t\frac{\partial u}{\partial t}&=u^2\Delta_{\bar\gamma_t}u-\frac{1}{2}R_{\bar\gamma_t}u^3+\frac{1}{2}\left(R_{\bar\gamma_t}-R_{\bar g}\right)u\\
\quad u(\cdot,0)&=\frac{\bar H_0}{H}>0.
\end{aligned}
\right.
\end{equation}
for $u(x,t)$ on $\Sigma\times [0,1]$. This equation has a local positive solution around $t=0$. By a similar analysis as in the Claim \ref{quassphericalestimate}, we can prove the solution decreases at most exponentially in time. Since $R_{\bar\gamma_t}\geq 0$, the coefficient of $u^3$ is nonpositive. By comparing $u$ with the solution to the corresponding ODE 
\begin{equation*}
\left\{
\begin{aligned}
\frac{\partial w}{\partial t}&=\frac{M}{2}w,\\
\quad w(0)&\equiv\max_\Sigma\frac{\bar H_0}{H},
\end{aligned}
\right.
\end{equation*}
where $M=\max_{\Sigma\times [0,1]}\bar H^{-1}_t\left(|R_{\gamma_t}|+|R_{\bar g}|\right)$, we find that $u$ grows at most exponentially in time. So the local positive solution can be extended to the whole $[0,1]$. 

Set $g=u^2 dt^2+\bar\gamma_t$. By Claim \ref{quassphericalscalar}, $R_g\equiv 0$. Let $A_t$ and $H_t$ denote the second fundamental form and the mean curvature of $\Sigma_t$ induced from metric $g$. It is not hard to see
\begin{equation}\label{meancurrelationsec3}
A_t=u^{-1}\bar A_t,\quad\  H_t=u^{-1}\bar H_t.
\end{equation}
By the second variation formula, Gauss equation and relation \eqref{meancurrelationsec3}, we have
\begin{equation*}
\begin{split}
\frac{d}{dt}\int_{\Sigma_t}H_t\,d\mu_{\bar\gamma_t}&=\int_{\Sigma_t}\left(H_t^2-\|A_t\|^2-\ric(\nu,\nu)\right)u\,d\mu_{\bar\gamma_t}\\
&=\frac{1}{2}\int_{\Sigma_t}\left(\bar H_t^2-\|\bar A_t\|^2\right)u^{-1}\,d\mu_{\bar \gamma_t}+\frac{1}{2}\int_{\Sigma_t}R_{\bar\gamma_t}u\,d\mu_{\bar\gamma_t},
\end{split}
\end{equation*}
where $\nu=u^{-1}\partial_t$ is the unit normal to $\Sigma_t$. Since $\bar H_t^2>\|\bar A_t\|^2$ and $R_{\bar\gamma_t}\geq 0$, $\frac{d}{dt}\int_{\Sigma_t}H_t\,d\mu_{\bar\gamma_t}>0$.
It follows that
$$\int_{\Sigma}H\, d\mu_{\gamma_0}=\int_{\Sigma}H_0\, d\mu_{\bar\gamma_0}<\int_{\Sigma}H_1\, d\mu_{\bar\gamma_1}.$$

Glue $(\Sigma\times [0,1], g)$ to $(\Omega,\tilde g)$ along their common boundary $(\Sigma,\gamma_0)$. Denote the obtained manifold by $(\hat\Omega,\hat g)$. It is not hard to see that $(\hat\Omega,\hat g)$ is a fill-in of Bartnik data $(\Sigma,\bar\gamma_1,H_1)$ but with conners along $(\Sigma,\gamma_0)$. Away from the corners, $\hat g$ is smooth and has nonnegative scalar curvature. On the corners, both the induced metrics and the mean curvatures from the two sides match. Then after performing Miao's mollifying procedure \cite{Miao02} and a suitable conformal deformation that preserves the boundary metric, we may get a family of smooth NNSC metrics $\{\hat g_\delta\}_{0<\delta\leq\delta_0}$ on $\hat\Omega$ for some $\delta_0>0$ such that $g_\delta|_{\partial\hat\Omega}=\bar\gamma_1$. Moreover, $\hat g_\delta$ locally uniformly converges to $\hat g$ away from the corners in the $C^{1,\,\alpha}$-sense. So the induced mean curvature of $\hat g_\delta$ on $\partial\hat\Omega\simeq\Sigma$, denoted by $\hat H_\delta$, converges to $H_1$ as $\delta\rightarrow 0$ in the $C^\alpha$-sense. Choose $\delta$ sufficiently small so that $\hat H_\delta>0$. Clearly, $(\hat\Omega,\hat g_\delta)$ is a NNSC fill-in of $(\Sigma,\bar\gamma_1,\hat H_\delta)$.  So by definition, 
$$\int_{\Sigma}\hat H_\delta\, d\mu_{\bar\gamma_1}\leq\Lambda_+(\Sigma,\bar\gamma_1).$$
Letting $\delta\rightarrow 0$, we get
$$\int_{\Sigma}H_1\, d\mu_{\bar\gamma_1}\leq\Lambda_+(\Sigma,\bar\gamma_1).$$
It follows that 
$$\int_{\Sigma}H\, d\mu_{\gamma_0}\leq\Lambda_+(\Sigma,\bar\gamma_1).$$
Because $(\Omega,\tilde g)$ is an arbitrary NNSC fill-in of $(\Sigma,\gamma_0,H)$ for arbitrary $H>0$, in fact we have
$$\Lambda_+(\Sigma,\gamma_0)\leq \Lambda_+(\Sigma,\bar\gamma_1).$$
Assume $\Sigma$ has dimension $n$. By the scaling property of $\Lambda_+$-invariant, $$\Lambda_+(\Sigma,\bar\gamma_1)=e^{(n-1)\varepsilon}\Lambda_+(\Sigma,\gamma_1).$$
Since $\varepsilon$ is an arbitrary positive constant, by letting $\varepsilon\rightarrow 0$, we obtain
$$\Lambda_+(\Sigma,\gamma_0)\leq \Lambda_+(\Sigma,\gamma_1).$$
\end{proof}

Now we are ready to prove Theorem \ref{pmt1}. 
\begin{proof}[Proof of Theorem \ref{pmt1}] If PMT holds for AF $n$-manifolds, then by nonincreasing property and convergence of  Brown-York mass along a quasi-spherical foliation constructed in \cite{ST02} (see Lemma 4.2 and Theorem 2.1), we have $$\Lambda_{+}\left(\mathbf{S}^{n-1}, \gamma_{\rm std}\right)\leq (n-1)\omega_{n-1}.$$ On the other hand, since $\left(\mathbf{S}^{n-1}, \gamma_{\rm std}\right)$ is the boundary of the unit ball in $\mathbf R^{n}$, we certainly have $$\Lambda_{+}\left(\mathbf{S}^{n-1}, \gamma_{\rm std}\right)\geq (n-1)\omega_{n-1}.$$ Thus the necessity part holds.

Due to Schoen-Yau's result (the statement in the bottom of p.48 in \cite{SY81}), it is sufficient to prove the sufficiency part for the AS case. Assume $(M^n,g)$ is an AS manifold with nonnegative scalar curvature and mass $m$. 
In an AS coordinates near infinity,
\begin{equation*}
g_{ij}=\left(1+\frac{2m}{n-2}r^{2-n}\right)\delta_{ij}+O(r^{1-n}),
\end{equation*}
where $r=|x|$. 
The Chrisitoffel symbol of $g$ is calculated as
\begin{equation*}
\Gamma^k_{ij}=m\left(\delta_{ij}x^k-\delta_{ik}x^j-\delta_{jk}x^i\right)r^{-n}+O(r^{-n}).
\end{equation*}
Simple calculations give $r_i=\frac{x^i}{r},$ and
\begin{align*}
r_{ij}=\left(\frac{1}{r}-\frac{m}{r^{n-1}}\right)\delta_{ij}-\left(\frac{1}{r^3}-\frac{2m}{r^{n+1}}\right)x^ix^j+O(r^{-n}).
\end{align*}
It follows that
\begin{equation*}
|\nabla_g r|=r^{-1}\sqrt{g^{ij}x^ix^j}=1-\frac{m}{n-2}r^{2-n}+O(r^{1-n}),
\end{equation*}
\begin{align*}
\partial_{k}|\nabla_g r|=\frac{mx^k}{r^n}+O(r^{-n}),
\end{align*}
and
\begin{align*}
\Delta_g r&=g^{ij}r_{ij}=\frac{n-1}{r}-\left(n+\frac{2}{n-2}\right)mr^{1-n}+O(r^{-n}).
\end{align*}
Then the mean curvature of $S_r$ in $(M^n,g)$ is
\begin{align*}
H_r=\diver_g\left(\frac{\nabla_g r}{|\nabla_g r|}\right)=\frac{n-1}{r}\left(1-\frac{n-1}{n-2}\frac{m}{r^{n-2}}\right)+O(r^{-n}).
\end{align*}
It is not hard to see that the induced metric $\gamma_r$ on $S_r$ satisfies
\begin{align*}
\gamma_r=r^2\left(1+\frac{2}{n-2}\frac{m}{r^{n-2}}\right)\left(\gamma_{\rm std}+\beta\right),
\end{align*}
where $\beta$ is a symmetric $2$-tensor on $\mathbf S^{n-1}$ and decays as $\beta=O(r^{1-n})$.
So the area of $S_r$ is
\begin{equation*}
\Area(S_r)=\omega_{n-1}r^{n-1}\left(1+\frac{n-1}{n-2}\frac{m}{r^{n-2}}\right)+O(1).
\end{equation*}
It follows that
\begin{equation}\label{areaasym}
\int_{S_r}H_r\, d\mu_{\gamma_r}=(n-1)\omega_{n-1}r^{n-2}+O(r^{-1}).
\end{equation}

By the scaling property of the $\Lambda_+$-invariant, we have
\begin{align*}
\Lambda_+(S_r,\gamma_r)&=r^{n-2}\left(1+\frac{2}{n-2}\frac{m}{r^{n-2}}\right)^{\frac{n-2}{2}}\Lambda_+(\mathbf S^{n-1},\gamma_{\rm std}+\beta).
\end{align*}
Since $|\beta|_{C^2(\mathbf S^{n-1})}$ can be arbitrarily small as $r\rightarrow\infty$, for sufficiently large $r$, $\{\gamma(s)\}_{s\in[0,1]}$ is smooth monotonically increasing path of PSC metrics, 
where
$$\gamma(s)=\left(1+s|\beta|_{L^\infty(\mathbf S^{n-1})}\right)\gamma_{\rm std}+(1-s)\beta.$$
Due to Lemma \ref{c0monotonicity}, we have
\begin{align*}
\Lambda_+(\mathbf S^{n-1},\gamma_{\rm std}+\beta)\leq\left(1+|\beta|_{L^\infty(\mathbf S^{n-1})}\right)^{\frac{n-2}{2}}\Lambda_+(\mathbf{S}^{n-1},\gamma_{\rm std}).
\end{align*}
Similarly, 
\begin{align*}
\Lambda_+(\mathbf S^{n-1},\gamma_{\rm std}+h)\geq\left(1-|\beta|_{L^\infty(\mathbf S^{n-1})}\right)^{\frac{n-2}{2}}\Lambda_+(\mathbf{S}^{n-1},\gamma_{\rm std}).
\end{align*}
Since $\beta=O(r^{1-n})$, it follows that
\begin{align}\label{lambdaasm1}
\Lambda_+(S_r,\gamma_r)=\left(r^{n-2}+m\right)\Lambda_+(\mathbf{S}^{n-1},\gamma_{\rm std})+O(r^{-1}).
\end{align}

Combining  \eqref{areaasym} and \eqref{lambdaasm1} together, we see that as $r\rightarrow\infty$,
$$
m_{\rm BY}(S_r,\gamma_r, H_r)=m+O(r^{-1}).
$$
By the definition of $\Lambda_{+}$-invariant and nonnegativity of the scalar curvature of $(M^n,g)$, we see that $m_{\rm BY}(S_r,\gamma_r,H_r)\geq 0$. Hence, $m\geq 0$. 
 
Once we prove the nonnegativity of the ADM mass for any AF manifold with nonnegative scalar curvature, the rigidity part can be obtained by usual deformation arguments. For instance, suppose $(M^n,g)$ is an AF manifold with nonnegative scalar curvature and vanishing ADM mass. If $(M^n,g)$ is not scalar-flat, then by a suitable conformal deformation, we may get an AF metric with nonnegative scalar curvature and negative ADM mass, which leads to a contradiction. If $(M^n,g)$ is scalar-flat but not Ricci-flat, then we can run Ricci flow on $(M^n,g)$ to get an AF metric with strictly positive scalar curvature and zero ADM mass (see \cite{DM}). Thus, $(M^n,g)$ must be Ricci-flat. Finally, by Corollary 6.7 in \cite{BKN}, we know that $(M^n,g)$ is isometric to $\mathbf{R}^n$. This completes the proof of Theorem \ref{pmt1}.
\end{proof}

\begin{rem}\label{gbylimitas}
As a by-product, by combining \eqref{areaasym} and \eqref{lambdaasm1} together, for an AS $n$-manifold with mass $m$, we have either
$$\lim_{r\rightarrow\infty}m_{\rm BY}(S_r,\gamma_r,H_r)=\infty,$$
or
$$\lim_{r\rightarrow\infty}m_{\rm BY}(S_r,\gamma_r,H_r)=m,$$
where $S_r$ is the $r$-coordinate sphere in some compatible coordinates, $\gamma_r$ and $H_r$ are the induced metric and mean curvature of $S_r$ respectively. And it is not hard to see that second case holds if and only if $\Lambda_{+}\left(\mathbf{S}^{n-1}, \gamma_{\rm std}\right)=(n-1)\omega_{n-1}$. This leads to another equivalent statement of the positive mass theorem for AF $n$-manifolds: there exists an AS $n$-manifold of which the large sphere limit of the generalized Brown-York mass is finite.
\end{rem}

In fact, for an arbitrary convex hypersurface in $\mathbf{R}^n$, the conclusion of Theorem \ref{pmt1} is still true. More precisely, let $(\Sigma^{n-1},\gamma)$ be a closed convex hypersurface in $\mathbf{R}^{n}$ and $H_0$ be its mean curvature. Denote the total mean curvature by
$$h\left(\Sigma^{n-1},\gamma\right)=\int_{\Sigma^{n-1}}H_0\,d\mu_{\gamma}.$$
Then we have 
\begin{thm}\label{pmt2}
Let $(\Sigma^{n-1},\gamma)$ be a closed convex hypersurface in $\mathbf{R}^{n}$. Then $\Lambda_+(\Sigma^{n-1},\gamma)=h(\Sigma^{n-1},\gamma)$ if and only if PMT holds for AF $n$-manifolds.\end{thm} 
\begin{proof}
The necessity part is the same as Theorem \ref{pmt1}. For the sufficiency part, by Theorem \ref{pmt1}, we only need to show that $\Lambda_+(\Sigma^{n-1},\gamma)=h(\Sigma^{n-1},\gamma)$ implies $\Lambda_+(\mathbf{S}^{n-1},\gamma_{\rm std})=(n-1)\omega_{n-1}$. We prove this by contradiction. Suppose $(\mathbf{S}^{n-1},\gamma_{\rm std},H)$ admits a NNSC fill-in $(\Omega,\tilde{g})$ for some smooth function $H>0$ on  $\mathbf{S}^{n-1}$ that satisfies
$$\int_{\mathbf{S}^{n-1}}H\,d\mu_{\gamma_{\rm std}}>(n-1)\omega_{n-1}.$$
As in \cite{ST02}, we can use the quasi-spherical metric to construct a scalar-flat AF end $(M_+,g_+)$ with inner boundary data $(\mathbf{S}^{n-1},\gamma_{\rm std},H)$ and $$m_{\rm ADM}\left(M_+,g_+\right)<0.$$ Next, we glue $(M_+,g_+)$ and $(\Omega,\tilde{g})$ together along their common boundary $(\mathbf{S}^{n-1},\gamma_{\rm std})$. After performing Miao's mollifying procedure \cite{Miao02} and a suitable conformal deformation, we can obtain a smooth AF manifold $(M,g)$ with $m_{\rm ADM}(M,g)<0$ and $R_{g}\geq0$. Then by the results in \cite{SY81}, there is a scalar-flat metric $g_1$ on $M$ with conformally flat asymptotics and $m_{\rm ADM}(M,g_1)<0$. It was observed by Lohkamp in \cite{L} that such $M$ admits a metric $g_2$ with the following properties: $R_{g_2}\geq0$; $g_2$ is Euclidean outside a compact set $K\subset M$; $R_{g_2}>0$ somewhere in $K$.  

Since $(M,g_2)$ is Euclidean outside $K$, we may find a constant $\lambda>0$ such that $(\Sigma^{n-1},\lambda^2\gamma)$ can be isometrically embedded into $(M\setminus K,g_2)$. Denote the image of this embedding by $\Sigma_0$ and the region enclosed by $\Sigma_0$ in $M$ by $K_0$. Obviously, $(K_0,g_2)$ is a NNSC fill-in of $(\Sigma^{n-1},\lambda^2\gamma,\lambda^{-1}H_0)$. Next we increase the mean curvature of $\Sigma_0$ but preserve the nonnegativity of the scalar curvature by a conformal deformation. Note that $R_{g_2}(x_0)>0$ for some $x_0\in\mathring{K_0}$.  
We can find a small $\rho>0$ such that $R_{g_2}>0$ in $B_{\rho}(x_0)\subset K_0$. Let $\eta$ be a smooth function compactly supported in $B_{\rho}(x_0)$ that satisfies
$0\leq\eta\leq 1$ and $\eta(x_0)=1$. Define
$f=\eta R_{g_2}$. Now consider the following equation
\begin{equation*}
\left\{
\begin{aligned}
\Delta_{g_2}u-c^{-1}_nfu&=0\quad\text{in}\ K_0,\\
u&=1\quad \text{on}\ \Sigma_0,
\end{aligned}
\right.
\end{equation*}
where $c_n=\frac{4(n-1)}{n-2}$. Since $f\geq 0$, above equation has a smooth solution $u$. By the maximum
principle, $0<u<1$ in $\mathring{K_0}$ and $\frac{\partial
u}{\partial\nu}\big |_{\Sigma_0}>0$, where $\nu$ is the outward unit normal with
respect to $g_2$.

Now, let $g_3=u^\frac{4}{n-2}g_2$. Then there holds
\begin{align*}
R_{g_3}&=u^{-\frac{n+2}{n-2}}\left(R_{g_2}u-c_n\Delta_{g_2}u\right)\\
&=u^{-\frac{4}{n-2}}(1-\eta)R_{g_2}\geq 0,
\end{align*}
and
\begin{equation}\label{increasemeancur}
H_{g_3}=\lambda^{-1}H_0+\frac{c_n}{2}\frac{\partial u}{\partial\nu}>\lambda^{-1}H_0,
\end{equation}
where $H_{g_3}$ is the mean curvature of $\Sigma_0$ in $(M,g_3)$ with respect to the outward direction. 
Inequlity \eqref{increasemeancur} yields
$$\Lambda_+(\Sigma^{n-1},\lambda^2\gamma)> h(\Sigma^{n-1},\lambda^2\gamma).$$
Rescaling back, we get $\Lambda_+(\Sigma^{n-1},\gamma)>h(\Sigma^{n-1},\gamma)$, which contradicts the assumption. This completes the proof.
\end{proof}
\begin{rem}
Recently, P. Miao told us that one can also prove Theorem \ref{pmt1} by combining the Lohkamp's reduction \cite{L} and the first variation of the total mean curvature \cite{MST}, which is in the same spirit of the proof of Theorem \ref{pmt2}.
\end{rem}

It is also interesting to see that Theorem \ref{pmt1} reveals a relationship between PMT for AH manifolds and PMT for AF manifolds. Namely,
\begin{thm}\label{ahafpmt}
If PMT holds for AH $n$-manifolds, then PMT also holds for AF $n$-manifolds.
\end{thm}
Before prove this theorem, wet establish a lemma. 
\begin{lm}\label{AHPMTconfillin}
If PMT holds for AH $n$-manifolds, then for any $\lambda>0$,  
\begin{equation*}
\Lambda_{+,\,-1}\left(\mathbf{S}^{n-1},\lambda^2 \gamma_{\rm std}\right)=(n-1)\omega_{n-1}\lambda^{n-2}\sqrt{1+\lambda^2}.
\end{equation*}
 \end{lm}
\begin{proof}
Since $(\mathbf{S}^{n-1},\lambda^2 \gamma_{\rm std})$ is the boundary of a geodesic ball of radius $\arcsinh\lambda$ in the hyperbolic space $\mathbf {H}^n$, we certainly have
\begin{equation}\label{lambda2}
\Lambda_{+,\,-1}\left(\mathbf{S}^{n-1},\lambda^2 \gamma_{\rm std}\right)\geq (n-1)\omega_{n-1}\lambda^{n-2}\sqrt{1+\lambda^2}.
\end{equation}
Next we prove that PMT for AH $n$-manifolds implies the reverse inequality. Suppose for some function $H>0$ on $\mathbf{S}^{n-1}$, $(\mathbf{S}^{n-1},\lambda^2\gamma_{\rm std},H)$ admits a fill-in $(\Omega,\tilde{g})$ with $R_{\tilde g}\geq-n(n-1)$. Set $r_0=\arcsinh\lambda$. Now consider the quasi-spherical metric equation on $\mathbf{S}^{n-1}\times[r_0,\infty)$: 
\begin{equation}\label{Eq: quasi spherical equation 1}
\left\{
\begin{aligned}
\sinh(2r)\frac{\partial u}{\partial r}&=\frac{2}{n-1}u^2\Delta_{\gamma_{\rm std}}u-\left(n\sinh^2r+n-2\right)\left(u^3-u\right)\\
\quad u(\cdot,r_0)&=\frac{(n-1)\coth r_0}{H}.
\end{aligned}
\right.
\end{equation}
By similar analysis as in the proof of Theorem 2.1 in \cite{WY}, equation \eqref{Eq: quasi spherical equation 1} has a positive solution on the whole $[r_0,\infty)$, and the solution satisfies 
\begin{equation}\label{expansion1}
u(\omega,r)=1+e^{-nr}v(\omega)+O\left(e^{-(n+1)r}\right) \quad \text{as \,$r\rightarrow \infty$},
\end{equation}
where $\omega$ is the coordinates on $\mathbf{S}^{n-1}$ and $v$ is a smooth function on $\mathbf{S}^{n-1}$. The standard metic on $\mathbf H^n$ can be written as
$g_0=dr^2+\sinh^2r\gamma_{\rm std}$. On $\mathbf{S}^{n-1}\times[r_0,\infty)$, set $g=u^2dr^2+\sinh^2r\gamma_{\rm std}$. Then by Claim \ref{quassphericalscalar}, $R_{g}\equiv-n(n-1)$. Let $\Sigma_r$ be the $r$-slice and $\gamma_r=\sinh^2r\gamma_{\rm std}$. Denote the induced mean curvature of $\Sigma_r$ from $g_0$ and $g$ by $H^0_r$ and $H_r$ respectively. Set $$ m(r)=\frac{1}{(n-1)\omega_{n-1}}\int_{\Sigma_r}\left(H^0_r-H_r\right)\cosh r\,d\mu_{\gamma_r}.$$
It is not hard to see
\begin{equation}\label{mass1}
m(r)=\sinh^{n-2}r\cosh^2 r\int_{\mathbf S^{n-1}}\left(1-u^{-1}(\omega,r)\right) d\omega.
\end{equation}
where $d\omega$ is the volume element of the standard metric. 
We have
\begin{align*}
 m'(r)=&\sinh^{n-3}r\cosh r\left(n-2+n\sinh^2r\right)\int_{\mathbf S^{n-1}}\left(1-u^{-1}\right) d\omega\\
&+\sinh^{n-2}r\cosh^2 r\int_{\mathbf S^{n-1}}u^{-2}\frac{\partial u}{\partial r}\, d\omega\\
=&\sinh^{n-3}r\cosh r\left(n-2+n\sinh^2r\right)\int_{\mathbf S^{n-1}}\left(1-u^{-1}\right) d\omega\\
&+\frac{1}{n-1}\sinh^{n-3}r\cosh r\int_{\mathbf S^{n-1}}\Delta_{\gamma_{\rm std}}u\,d\omega\\
&+\frac{1}{2}\sinh^{n-3}r\cosh r\int_{\mathbf S^{n-1}}\left(n-2+n\sinh^2r\right)\left(u^{-1}-u\right)d\omega\\
=&-\frac{1}{2}\sinh^{n-3}r\cosh r\int_{\mathbf S^{n-1}}u^{-1}\left(u-1\right)^2\,d\omega\leq 0.
\end{align*}
So $ m(r)$ monotonically decreases.

Form \eqref{expansion1} and \eqref{mass1}, we see that
\begin{equation}\label{limitmass}
\lim_{r\rightarrow\infty}m(r)=\frac{1}{2^n}\int_{\mathbf{S}^{n-1}}v\,d\omega.
\end{equation}
Let $\rho=\log \frac{e^r+1}{e^r-1}$.
Then in a coordinate system near infinity,
\begin{equation*}
\begin{split}
g&=u^2dr^2+\sinh^2r\gamma_{\rm std}\\
&=u^2 \sinh^{-2}\rho\left(d\rho^2+ u^{-2}\gamma_{\rm std}\right).
\end{split}
\end{equation*}

Combining \eqref{expansion1} and Lemma 6.5 in \cite{BQ}, we see that there is a geodesic defining function $s$ for $g$ such that
$$
g=\sinh ^{-2}\hat{\rho}\left(d\hat{\rho}^{2}+\gamma_{\rm std}+\frac{2^{1-n}}{n}\hat{\rho}^{n} v\gamma_{\rm std}+O\left(\hat{\rho}^{n+1}\right)\right),$$
where
$$\hat\rho=\log\frac{1+s}{1-s}.$$
Then $(\mathbf S^{n-1}\times[r_0, \infty),g)$ is an AH end. Glue $(\Omega, \tilde g)$ and $(\mathbf S^{n-1}\times[r_0, \infty),g)$ together along their common boundary $(\mathbf{S}^{n-1},\lambda^2\gamma_{\rm std})$. The resulting manifold is AH with corners; its scalar curvature is at least $-n(n-1)$ away from the corners; and mean curvatures from the two sides of corners are equal. From Theorem 1.1 and its proof in \cite{BQ}, we see if PMT holds for smooth AH $n$-manifolds, then it also holds for AH $n$-manifolds with corners that have matching mean curvatures from two sides. Therefore, we have
\begin{equation}\label{massaspctintegral}
\int_{\mathbf{S}^{n-1}}v\,d\omega\geq 0.
\end{equation}

Together with the monotonicity of $m(r)$ and \eqref{limitmass}, \eqref{massaspctintegral} implies $m(r_0)\geq 0$, namely
\begin{equation*}
\int_{\mathbf S^{n-1}}H\,d\mu_{\lambda^2\gamma_{\rm std}}\leq (n-1)\omega_{n-1}\lambda^{n-2}\sqrt{1+\lambda^2}.
\end{equation*}
Since $(\Omega,\tilde{g})$ is an arbitrary fill-in of $(\mathbf{S}^{n-1},\lambda^2\gamma_{\rm std},H)$ for an arbitrary function $H>0$ and $R_{\tilde g}\geq-n(n-1)$, by definition we have 
\begin{equation*}
\Lambda_{+,\,-1}\left(\mathbf{S}^{n-1},\lambda^2 \gamma_{\rm std}\right)\leq (n-1)\omega_{n-1}\lambda^{n-2}\sqrt{1+\lambda^2}.
\end{equation*}
This completes the proof.
\end{proof}
Now we prove Theorem \ref{ahafpmt}. 
\begin{proof}[Proof of Theorem \ref{ahafpmt}]
By the scaling property of $\Lambda_{+, \,\kappa}$-invariant, for any $\kappa<0$,
$$\Lambda_{+, \,\kappa}\left(\mathbf{S}^{n-1},\gamma_{\rm std}\right)=|\kappa|^{1-\frac{n}{2}}\Lambda_{+,\,-1}\left(\mathbf{S}^{n-1},|\kappa|\gamma_{\rm std}\right).$$	
Due to Lemma \ref{AHPMTconfillin}, 
\begin{equation*}
\Lambda_{+,\,-1}\left(\mathbf{S}^{n-1},|\kappa|\gamma_{\rm std}\right)=(n-1)\omega_{n-1}|\kappa|^{\frac{n}{2}-1}\sqrt{1-\kappa}.
\end{equation*}
So for all $\kappa<0$, there holds
\begin{equation*}
\Lambda_{+, \,\kappa}\left(\mathbf{S}^{n-1},\gamma_{\rm std}\right)=(n-1)\omega_{n-1}\sqrt{1-\kappa}.
\end{equation*}
By definition, $\Lambda_+(\mathbf{S}^{n-1}, \gamma_{\rm std})\leq \Lambda_{+, \,\kappa}\left(\mathbf{S}^{n-1},\gamma_{\rm std}\right)$. 
Hence, 
\begin{equation}\label{asymptotic}
\Lambda_+(\mathbf{S}^{n-1}, \gamma_{\rm std})\leq (n-1)\omega_{n-1}\sqrt{1-\kappa}.
\end{equation}
Letting $\kappa\rightarrow 0$ in \eqref{asymptotic}, we get
$$
\Lambda_+(\mathbf{S}^{n-1}, \gamma_{\rm std})\leq (n-1)\omega_{n-1}.
$$
Since $(\mathbf{S}^{n-1}, \gamma_{\rm std})$ is the boundary of the unit ball in $\mathbf{R}^n$, 
$$
\Lambda_+(\mathbf{S}^{n-1}, \gamma_{\rm std})\geq (n-1)\omega_{n-1}.
$$
Thus, we finally obtain
$$
\Lambda_+(\mathbf{S}^{n-1}, \gamma_{\rm std})= (n-1)\omega_{n-1}.
$$
By Theorem \ref{pmt1}, we see that PMT holds for AF manifolds. Thus, we completes the proof of Theorem \ref{ahafpmt}.
\end{proof}

%%-----------------------------------------------------------------------------------------------------------------------------------------------------------------------------------------------------------------------------------------------------------------------------------

\section{Total mean curvature estimate for spin fill-ins of spheres}\label{section4}
In this section, we consider the upper bound for the boundary total mean curvature of a compact spin manifold with NNSC. We say $\left( \Omega^n, g\right)$ is a spin fill-in for Bartnik data $\left(\Sigma^{n-1}, \gamma, H\right)$ if $\Omega^n$ is a spin manifold and $\left( \Omega^n, g\right)$ is a fill-in of $\left(\Sigma^{n-1}, \gamma, H\right)$.

\begin{thm}\label{Gconjec1}
For $n \geq 3$, let $\gamma$ be a smooth Riemannian metric on $\mathbf{S}^{n-1}$. Then  there is a constant $h_0=h_0(n,\gamma)$ such that if $H$ is a smooth positive function on $\mathbf S^{n-1}$ and $\left(\mathbf{S}^{n-1}, \gamma, H\right)$ admits a spin NNSC fill-in, then
$$ \int_{\mathbf{S}^{n-1}} H\, d \mu_{\gamma}\leq h_{0}.
$$
\end{thm}
\begin{proof}
There is a constant $\lambda>0$ such that $\lambda^2\gamma_{\rm std}>\gamma$. 
Let $\gamma_t= (1-t)\gamma+t\lambda^2\gamma_{\rm std}$ for  $t$ in $[0,1]$. Then $\gamma_t$ strictly monotonically increases. On $\mathbf{S}^{n-1}\times [0,1]$, define $\bar g=dt^2+\gamma_t$. Denote $\mathbf{S}^{n-1}\times\{t\}$ by $\Sigma_t$. Let $\bar A_t$ and $\bar H_t$ be the second fundamental form and mean curvature of $\Sigma_t$ induced from $\bar g$ with respect to the $\partial_t$-direction. Since $\gamma_t$ strictly monotonically increases, $\bar A_t>0$. It follows that
 $\bar H_t>0$ and $\bar H_t^2-\|\bar A_t\|^2>0$. Set $K=\min_{\mathbf{S}^{n-1}\times [0,1]}R_{\gamma_t}$. We see that $K$ depends only on $n$ and $\gamma$. Without loss of generality, we assume $K<0$. Suppose $(\Omega,\tilde{g})$ is a spin NNSC fill-in of $(\mathbf{S}^{n-1},\gamma,H)$. Now consider the quasi-spherical metric equation $\mathbf{S}^{n-1}\times [0,1]$: 
 \begin{equation}\label{Eq: quasi spherical equation 4}
\left\{
\begin{aligned}
\bar H_t\frac{\partial u}{\partial t}&=u^2\Delta_{\gamma_t}u+\frac{1}{2}\left(K-R_{\gamma_t}\right)u^3+\frac{1}{2}\left(R_{\gamma_t}-R_{\bar g}\right)u\\
\quad u(\cdot,0)&=\frac{\bar H_0}{H}>0.
\end{aligned}
\right.
\end{equation}
Since the coefficient of $u^3$ is $K-R_{\gamma_t}\leq 0$, after a similar analysis as for equation \eqref{Eq: quasi spherical equation 2}, we can prove above equation has a positive solution on the whole $[0,1]$. 

On $\mathbf{S}^{n-1} \times [0,1]$, let $g=u^2dt^2+\gamma_t$. By Claim \ref{quassphericalscalar}, we see $R_g\equiv K$. Let $A_t$ and $H_t$ denote the second fundamental form and the mean curvature of $\Sigma_t$ induced from metric $g$. It is not hard to see
\begin{equation}\label{meancurrelation}
A_t=u^{-1}\bar A_t,\qquad H_t=u^{-1}\bar H_t.
\end{equation}
By the second variation formula, Gauss equation and relation \eqref{meancurrelation}, 
\begin{align*}
\frac{d}{dt}\int_{\Sigma_t}H_t\, d\mu_{\gamma_t}&=\int_{\Sigma_t}\left(H_t^2-\|A_t\|^2-\ric(\nu,\nu)\right)u\,d\mu_{\gamma_t}\\
&=\frac{1}{2}\int_{\Sigma_t}\left(\bar H_t^2-\|\bar A_t\|^2\right)u^{-1}\,d\mu_{\gamma_t}+\frac{1}{2}\int_{\Sigma_t}\left(R_{\gamma_t}-K\right)u\,d\mu_{\gamma_t},
\end{align*}
where $\nu=u^{-1}\partial_t$ is the unit normal to $\Sigma_t$. Since $\bar H_t^2>\|\bar A_t\|^2$ and $R_{\gamma_t}\geq K$, $\frac{d}{dt}\int_{\Sigma_t}H_t\,d\mu_{\gamma_t}>0$.
It follows that
\begin{equation}\label{finitespinmonotonicity}
\int_{\mathbf S^{n-1}}H\,d\mu_{\gamma}=\int_{\Sigma_0}H_0\, d\mu_{\gamma_0}<\int_{\Sigma_1}H_1\,d\mu_{\gamma_1}.
\end{equation}

Glue $(\Omega,\tilde g)$ and $(\mathbf S^{n-1}\times [0,1],g)$ together along their common boundary $(\mathbf S^{n-1},\gamma)$ and denote the new manifold by $(\Omega',g')$. Then $(\Omega',g')$ is a spin fill-in of $(\mathbf S^{n-1},\lambda^2\gamma_{\rm std},H_1)$with corners. It is not hard to see that $g'$ is Lipschitz across the corners and the mean curvatures from two sides of the corners match. On $\Omega'$ away from the corners, $R_{g'}\geq K$. 

We can construct an AH end with constant scalar curvature $-n(n-1)$ and inner boundary data $\big(\mathbf S^{n-1},-\frac{K\lambda^2}{n(n-1)}\gamma_{\rm std},\sqrt{-\frac{n(n-1)}{K}}H_1\big)$ as in the proof of Lemma \ref{AHPMTconfillin}. Then we glue $\big(\Omega',-\frac{K\lambda^2}{n(n-1)}g'\big)$ to this AH end to get a complete spin AH manifold with corners. This manifold has scalar curvature as least $-n(n-1)$ away from the corners; and the mean curvatures on the two sides of the corners are equal. Note that the positive mass theorem holds in this setting \cite{BQ}. Hence, by a very similar argument as Lemma \ref{AHPMTconfillin} and a scaling we get
\begin{equation*}
\int_{\Sigma_1}H_1\,d\mu_{\lambda^2\gamma_{\rm std}}\leq (n-1)\omega_{n-1}\lambda^{n-2}\sqrt{1-\frac{K\lambda^2}{n(n-1)}}. 
\end{equation*}
It follows from \eqref{finitespinmonotonicity} that
\begin{equation*}
\int_{\mathbf{S}^{n-1}} H\, d \mu_{\gamma}<(n-1)\omega_{n-1}\lambda^{n-2}\sqrt{1-\frac{K\lambda^2}{n(n-1)}}. 
\end{equation*}
This completes the proof. 
\end{proof}

\begin{rem}
When the assumption $R\geq 0$ is replaced by $R\geq\sigma$ for any constant $\sigma$, the conclusion of Theorem \ref{Gconjec1} still holds, and the argument of the proof works too. \end{rem}

\begin{rem}\label{hcd}
Due to Theorem 1 in \cite{EHLS} and Theorem 1.4 in \cite{CD}, Theorem \ref{Gconjec1} holds without the spin assumption for $3\leq n\leq7$.
\end{rem}

\section {Estimates for $\Lambda_+(\mathbf{S}^2, \gamma)$}\label{section5}

In this section, we investigate some properties of $\Lambda_+$-invariant for $2$-spheres with nonnegative Gauss curvature.  

We first give upper and lower bounds estimates for $\Lambda_+(\mathbf{S}^2, \gamma)$ in terms of the diameter. 
\begin{pro}\label{2dboundedness}
Suppose $\gamma$ is a smooth metric on $\mathbf{S}^2$ with Gauss curvature $K_\gamma\geq 0$. Then 
\begin{equation}\label{totalmeancurradius}
2\diam\left(\mathbf{S}^2,\gamma\right)<\Lambda_+\left(\mathbf{S}^2,\gamma\right)<12\pi\diam\left(\mathbf{S}^2,\gamma\right).
\end{equation}
\end{pro}
Then we prove the $C^0$ continuity of $\Lambda_+(\mathbf{S}^2,\gamma)$ with respect to $\gamma$ in the class of nonnegatively curved metrics. To state more concisely, we introduce the following notion.
\begin{defn}
Let $\Sigma$ be a smooth manifold and $\gamma_1$, $\gamma_2$ be two metrics on $\Sigma$. Define the dilation between $\gamma_1$ and $\gamma_2$ by
\begin{align*}
\dial(\gamma_1,\gamma_2)=\inf\left\{\lambda\,\bigg |\,\begin{array}{c}
\text{there is a diffeomorphism $\phi$ such}\\ \text{that }\  \lambda^{-1}\gamma_{2}\leq\phi^*(\gamma_1)\leq\lambda\gamma_{2}
\end{array}
\right\}.
\end{align*}
When $\Sigma$ is diffeomorphic to spheres with the canonical differential structure, $\dial(\gamma,\gamma_{\rm std})$ is abbreviated as $\dial(\gamma)$.
\end{defn}
It is clear that $\dial(\cdot,\cdot)$ is symmetric and invariant under diffeomorphisms.  The $C^0$ continuity of $\Lambda_+(\mathbf{S}^2,\gamma)$ is the following result: 

\begin{thm}\label{2dstability}
Let $\gamma_0$ be a smooth metric on $\mathbf S^2$ with $K_{\gamma_0}\geq0$. Then for any $\varepsilon>0$, there is a $\delta=\delta(\varepsilon,\gamma_0)>1$ such that for any
metric $\gamma$ with $K_\gamma\geq 0$ and $\dial(\gamma,\gamma_0)\leq\delta$,  
$$
\big|\Lambda_+\left(\mathbf S^2,\gamma\right)-\Lambda_+\left(\mathbf S^2,\gamma_0\right)\big|\leq\varepsilon.
$$
\end{thm}
Note that Propsition \ref{2dboundedness} and Theorem \ref{2dstability} jointly imply that Conjecture \ref{conj1} is true for nonnegatively curved $2$-spheres under the positive mean curvature restriction. 

To prove these two theorems, we first seek a more explicit expression for $\Lambda_+$-invariant of nonnegatively curved $\left(\mathbf S^2,\gamma\right)$. When $K_\gamma>0$, due to the classical results by Nirenberg  \cite{N} and Pogogolev \cite{P}, $(\mathbf{S}^2,\gamma)$ can be smoothly isometrically embedded in $\mathbf R^3$ as a convex surface. In this case, by the result of L.-F. Tam and the first author \cite{ST02}, 
\begin{equation}\label{positivelycurved}
\Lambda_+\left(\mathbf{S}^2,\gamma\right)=\int_{\mathbf{S}^2}H_0\,d\mu_{\gamma},
\end{equation}
where $H_0$ is the mean curvature of the embedding of $(\mathbf{S}^2,\gamma)$ with respect to the outward normal. When $K_\gamma$ is only nonnegative, by Guan-Li's result \cite{GL} and Hong-Zuily's result \cite{Hong}, $(\mathbf{S}^2,\gamma)$ can be $C^{1,1}$ isometrically embedded in $\mathbf{R}^3$ as a convex surface with mean curvature $H_0$ defined almost everywhere. According to Porogolev's rigidity theorem \cite[P.167, Theorem 1]{P}, $\int_{\mathbf{S}^2}H_0\,d\mu_{\gamma}$ is well-defined in the sense that it is the same for any $C^{1,1}$ isometric embeddings. Intuitively, \eqref{positivelycurved} should also holds for the nonnegatively curved case. Indeed, by Theorem 0.2 in \cite{ST04}, for such $(\mathbf{S}^2,\gamma)$, there holds
$$\Lambda_+\left(\mathbf{S}^2,\gamma\right)\leq\int_{\mathbf{S}^2}H_0\,d\mu_{\gamma}.$$
But we do not know whether $\Lambda_+$ can be achieved in this case since the isometric embedding of $(\mathbf{S}^2,\gamma)$ is only $C^{1,1}$. So we can't obtain the equality directly from the rigidity part of Theorem 1.2 in \cite{ST19}. We turn to use Lemma \ref{c0monotonicity} and the following monotonicity lemma for total mean curvature of convex surfaces.

\begin{lm}\label{monoquermassintegral}
Let $\Sigma_1$ and $\Sigma_2$ be two $C^{1,1}$ closed convex surfaces in $\mathbf{R}^{3}$. Suppose $\Sigma_1$ is enclosed by $\Sigma_2$. Then 
\begin{equation*}
|\Sigma_1|\leq |\Sigma_2|,\quad \text{and}\quad\int_{\Sigma_1}H\,d\mu\leq \int_{\Sigma_2}H\,d\mu,
\end{equation*}
where $H$ denotes the mean curvature of $\Sigma_1$ or $\Sigma_2$ with respect to the outward normal. As a corollary, if $\{\Sigma_i\}$ is a sequence of $C^{1,1}$ closed convex surfaces that converges to a $C^{1,1}$ closed convex surface $\Sigma_{\infty}$ in the Hausdorff distance, then both the area and total mean curvature of $\Sigma_i$ converge to that of $\Sigma_{\infty}$. 
\end{lm}

This monotonicity lemma will be used several times in this section. It can be seen quickly from the geometric probabilistic explanation of quermassintegral (see Chapters 13--14 in \cite{San}), but this requires certain knowledge of integral geometry and convex geometry. For the sake of completeness, we write a more ``geometric analytic" proof here, which might be known before.

\begin{proof}[Proof of Lemma \ref{monoquermassintegral}]
We first consider the smooth case. For $\rho\geq 0$, let $\Sigma^\rho_i$ be the set of points lying in the exterior region of $\Sigma_i$ in $\mathbf R^3$ with distance $\rho$ to $\Sigma_i$ $(i=1,\,2)$. Then both $\Sigma^\rho_1$ and $\Sigma^\rho_2$ are smooth convex surfaces. 
Denote the distance function to $\Sigma_1$ by $d_1$. For any point $p\in\Sigma^\rho_1$,  $(\Delta d_1)(p)=H_{\Sigma_1^\rho}(p)\geq 0$, where $H_{\Sigma_1^\rho}(p)$ denotes the mean curvature of $\Sigma_1^\rho$ at $p$. So $\Delta d_1\geq 0$ in the exterior region of $\Sigma_1$. Let $\Omega$ be the region enclosed by $\Sigma_2$ and $\Sigma_1$.   By the divergence theorem, we have 
\begin{equation*}
\int_{\Sigma_2}\frac{\partial d_1}{\partial \nu}\,d\mu-\int_{\Sigma_1}\frac{\partial d_1}{\partial \nu}\,d\mu=\int_{\Omega}\Delta d_1\,d\mu\geq 0.
\end{equation*}
It is obvious that $\frac{\partial d_1}{\partial \nu}\equiv 1$ on $\Sigma_1$ and $0<\frac{\partial d_1}{\partial \nu}\leq 1$ on $\Sigma_2$. Hence, $|\Sigma_2|\geq|\Sigma_1|$. 
 By the second variation formula and Gauss-Bonnet formula, we have 
 the expansion for the area of $\Sigma^\rho_i$, namely the Steiner's formula in dimension $2$: 
\begin{equation*}
\left|\Sigma^\rho_i\right|=\left|\Sigma_i\right|+\left(\int_{\Sigma_i}H\,d\mu\right)\rho+4\pi\rho^2.
\end{equation*}

Since $\Sigma_1$ is enclosed by $\Sigma_2$, $\Sigma^\rho_1$ is also enclosed by $\Sigma^\rho_2$. Similarly, we have $|\Sigma^\rho_1|\leq |\Sigma^\rho_2|$ for all $\rho\geq 0$.
Because the leading terms of the expansions of $|\Sigma^\rho_1|$ and $|\Sigma^\rho_2|$ are the same, the coefficients of the second terms must satisfy
\begin{equation}\label{monototalmeancureq}
\int_{\Sigma_1}H\,d\mu\leq \int_{\Sigma_2}H\,d\mu.
\end{equation}
Thus the lemma is proved for the smooth case. 

Next we use approximations to deal with the non-smooth case. By running the mean curvature flow for a short time and scaling homothetically slightly if needed, we can get smooth convex surface $\Sigma_{i,\,\varepsilon}$ that converges to $\Sigma_{i}$ in $C^{1,1}$ as $\varepsilon\rightarrow 0$ for $i=1$, $2$; and $\Sigma_{1,\,\varepsilon}$ is enclosed by $\Sigma_{2,\,\varepsilon}$. 
Then we have
\begin{equation*}
\int_{\Sigma_{1,\,\varepsilon}}H\,d\mu\leq \int_{\Sigma_{2,\,\varepsilon}}H\,d\mu.
\end{equation*}
Since the convergence is in $C^{1,1}$, by letting $\varepsilon\rightarrow 0$, we also get \eqref{monototalmeancureq}. 

If $\{\Sigma_i\}$ is a sequence of $C^{1,1}$ closed convex surfaces that converges to a $C^{1,1}$ closed convex surface $\Sigma_{\infty}$ in the Hausdorff distance, assuming the origin is enclosed by $\Sigma_{\infty}$,  we can find a sequence of constants $\{\lambda_i\}$ that $\lambda_i\geq 1$ and $\lambda_i\rightarrow 1$, such that $\lambda_i^{-1}\Sigma_{\infty}$ is enclosed by $\Sigma_i$ and $\Sigma_i$ is enclosed by $\lambda_i\Sigma_{\infty}$, where $\lambda_i^{-1}\Sigma_{\infty}$ means homothetically scaling $\Sigma_{\infty}$  by $\lambda_i^{-1}$. Then the convergence of area and total mean curvature follows from the monotonicity. This completes the proof.
\end{proof}

Now, we are ready to show that $\Lambda_+(\mathbf{S}^2,\gamma)$ equals to the total mean curvature of the $C^{1,1}$ embedding of $(\mathbf{S}^2,\gamma)$ in $\mathbf{R}^3$ when $K_{\gamma}\geq0 $.

\begin{thm}\label{c111}
Let $\gamma$ be a smooth metric on $\mathbf{S}^2$ with $K_{\gamma}\geq0$. Then
$$\Lambda_+\left(\mathbf{S}^2,\gamma\right)=\int_{\mathbf{S}^2}H_0\,d\mu_{\gamma},$$
where $H_0$ is the mean curvature of the $C^{1,1}$ isometric embedding of $(\mathbf{S}^2,\gamma)$ in $\mathbf R^3$ with respect to the outward normal. 
\end{thm}

\begin{proof}
By the proof of Lemma 3.1 in \cite{Hong}, we can find a smooth function $v$ on $\mathbf{S}^2$ such that $e^{2v}\gamma$ has positive Gauss curvature. Then for any $s>0$, 
$\gamma_s=e^{2sv}\gamma$ has Gauss curvature 
$$K_{\gamma_s}=se^{2(1-s)v}K_{\gamma_1}+(1-s)e^{-2sv}K_{\gamma}>0.$$
Set $\overline v=\max_{\mathbf S^2} v$, $\underline v=\min_{\mathbf S^2} v$. 
Obviously, for any $\varepsilon>0$, both $\{e^{2\varepsilon t\left(v-\underline v\right)}\gamma\}_{t\in [0,1]}$ and $\{e^{2\varepsilon(1-t)\left(v-\overline v\right)}\gamma\}_{t\in [0,1]}$ are paths of monotonically increasing NNSC metrics. By Lemma \ref{c0monotonicity}, there holds
\begin{equation*}
 e^{\varepsilon\underline v}\Lambda_{+}\left(\mathbf{S}^2,\gamma\right)\leq \Lambda_{+}\left(\mathbf{S}^2,\gamma_{\varepsilon}\right)\leq e^{\varepsilon\overline v} \Lambda_{+}\left(\mathbf{S}^2,\gamma\right)
\end{equation*}
It follows that
\begin{equation}\label{app1}
\lim_{\varepsilon\rightarrow0^+} \Lambda_{+}\left(\mathbf{S}^2,\gamma_\varepsilon\right)=\Lambda_{+}\left(\mathbf{S}^2,\gamma\right).
\end{equation}
Since $K_{\varepsilon}>0$, by \cite{N,P}, we can embed $\left(\mathbf{S}^2,\gamma_\varepsilon\right)$ smoothly into $\mathbf{R}^3$ as a strictly convex surface. Denote this convex surface by $\Sigma_\varepsilon$ and its mean curvature by $H_{\varepsilon,\,0}$. Then by \cite{ST02}, we know that
\begin{equation}\label{app2} 
\Lambda_{+}\left(\mathbf{S}^2,\gamma_\varepsilon\right)=\int_{\Sigma_\varepsilon} H_{\varepsilon,\,0}\,d\mu_{\gamma_\varepsilon}.
\end{equation}
It was shown in \cite{Hong} that for a subsequence $\varepsilon_i\rightarrow 0$, $\Sigma_{\varepsilon_i}$ converges to a $C^{1,1}$ embedding of $(\mathbf{S}^2,\gamma)$ in $C^{1,\,\alpha}$ sense (for any $\alpha\in(0,1)$). Then by Lemma \ref{monoquermassintegral}, we obtain
\begin{equation}\label{app3}
\lim_{i\rightarrow\infty}\int_{\Sigma_{\varepsilon_i}} H_{\varepsilon_i,\,0}\,d\mu_{\gamma_{\varepsilon_i}}=\int_{\mathbf{S}^2}H_0\,d\mu_{\gamma}.
\end{equation}
Combining \eqref{app1}, \eqref{app2} and \eqref{app3} together, we finally get
$$\Lambda_+\left(\mathbf{S}^2,\gamma\right)=\int_{\mathbf{S}^2}H_0\,d\mu_{\gamma}.$$
This completes the proof.
\end{proof}
Once this is done, we can exploit some ideas and techniques from convex geometry. Now we prove Proposition \ref{2dboundedness}. 
\begin{proof}[Proof of Propsition \ref{2dboundedness}]
Denote the image of $(\mathbf{S}^2,\gamma)$ when $C^{1,1}$ isometric embedded in $\mathbf R^3$ by $\Sigma$. By Theorem \ref{c111}, $\Lambda_+\left(\mathbf{S}^2,\gamma\right)$ is equal to the total mean curvature of $\Sigma$. 

We first prove the upper bound estimate. Suppose that $p,\,q\in\Sigma$ realize $\diam(\gamma)$. Let $\overline{pq}$ be the segment connecting $p$ and $q$ in $\mathbf R^3$ and $|\overline{pq}|$ be its length. Denote the midpoint of $\overline{pq}$ by $o$. For any $x\in \Sigma$, we must have $d_\gamma(x,p)\leq\diam(\gamma)$. Then it follows that
 \begin{align*}
|\overline{xo}|&< |\overline{xp}|+|\overline{po}|\leq d_\gamma(x,p)+|\overline{po}|\\
&\leq \diam(\gamma)+\frac{1}{2}\diam(\gamma)=\frac{3}{2}\diam(\gamma).
\end{align*}
This means that $\Sigma$ is strictly contained in the ball centered at $o$ with radius $R=\frac{3}{2}\diam(\gamma)$. By Lemma \ref{monoquermassintegral}, the total mean curvature of $\Sigma$ is strictly less than the total mean curvature of the sphere of radius $R$, namely
$$\Lambda_+\left(\mathbf S^2,\gamma\right)<\Lambda_+\left(\mathbf S^2,R^2\gamma_{\rm std}\right)=12\pi\diam(\gamma).$$

Next, we prove the lower bound estimate. Suppose $p',\,q'\in\Sigma$ realize the extrinsic diameter of $\Sigma$, denoted by $l$. The first step is proving 
\begin{equation}\label{insegmentestimate}
l>\frac{\diam(\gamma)}{\pi}.
\end{equation}
Let $o'$ denote the midpoint of $\overline{p'q'}$. Since $\Sigma$ is convex, $o'$ lies in the interior of $\Sigma$. Then for any $x\in \Sigma$, we must have $|\overline{xo'}|<l$. Otherwise, the extrinsic diameter of $\Sigma$ is strictly greater than $l$. So $\Sigma$ is strictly contained in the ball centered at $o'$ with radius $l$. Take a plane $P$ that passes through $\overline{pq}$ and an arbitrary interior point of the convex body enclosed by $\Sigma$. Since $\Sigma$ is a closed $C^{1,1}$ convex surface, $\Sigma$ and $P$ intersect transversely. Denote the intersection curve by $\Gamma$. For $\Gamma$ passes through $p$, $q$ and $\Gamma\subset\Sigma$, its length $|\Gamma|\geq 2\diam(\gamma)$. Since $\Gamma$ is a planar convex curve enclosed by a circle of radius $l$, by one dimensional version of Lemma \ref{monoquermassintegral}, $|\Gamma|<2\pi l$. Thus we have proved $l>\diam(\gamma)/\pi$. For small $\varepsilon_1>0$, take $s$, $t\in\overline{p'q'}$ such that $|\overline{sp'}|=|\overline{tq'}|=\varepsilon_1$.  Take small $\varepsilon_2<\varepsilon_1$ such that the cylinder with axis $\overline{st}$ and radius $\varepsilon_2$ is enclosed by $\Sigma$. Cap the cylinder with two hemispheres of radius $\varepsilon_2$. Denote the combined surface by $C$. Apparently, $C$ is a $C^{1,1}$ convex surface. We can choose $\varepsilon_2$ small enough so that $C$ is enclosed by $\Sigma$. The total mean curvature of $C$ is  
$$\int_{C}H\,d\mu=\frac{1}{\varepsilon_2}\times 2\pi \varepsilon_2\times(l-2\varepsilon_1)+2\times 4\pi\varepsilon_2>2\pi(l-2\varepsilon_1).$$
Because $C$ is enclosed by $\Sigma$, by Lemma \ref{monoquermassintegral}, the total mean curvature of $\Sigma$ is not less than $2\pi(l-2\varepsilon_1)$. Since $\varepsilon_1$ can be arbitrarily small, in fact the total mean curvature is not less than $2\pi l$. Finally, due to \eqref{insegmentestimate}, we get the lower bound estimate. 
\end{proof}

\begin{rem}\label{GHrem}
There should be different proofs for \eqref{totalmeancurradius} type estimates. The constants in \eqref{totalmeancurradius} might not be sharp, but are enough for our purpose. In \cite{Topping2}, P. Topping proved a much more general result for the lower bound estimate, but with a larger constant when in the particular setting of Proposition \ref{2dboundedness}.
\end{rem}

Now we are in a position to prove Theorem \ref{2dstability}. 
\begin{proof}[Proof of Theorem \ref{2dstability}] We prove this result by contradiction. If the conclusion is not true, then there exists an $\varepsilon_0>0$, and a sequence of metrics $\gamma_i$ on $\mathbf{S}^2$ with $K_{\gamma_i}\geq0$ and $\dial(\gamma_i,\gamma_0)\rightarrow 1$ such that
$$\big|\Lambda_+\left(\mathbf S^2,\gamma_i\right)-\Lambda_+\left(\mathbf S^2,\gamma_0\right)\big|\geq \varepsilon_0.$$
Since $\dial\left(\gamma_i,\gamma_0\right)\rightarrow1$, by Theorem 7.3.25 in \cite{BBI}, we have
\begin{equation}\label{gh}
\left(\mathbf{S}^2,\gamma_i\right)\xrightarrow{d_{\rm GH}}\left(\mathbf{S}^2,\gamma_0\right),
\end{equation}
where $d_{\rm GH}$ denotes the Gromov-Hausdorff distance. By $\dial\left(\gamma_i,\gamma_0\right)\rightarrow1$, we also have 
\begin{equation}\label{diameterconverge}
\diam\left(\mathbf{S}^2, \gamma_i\right)\rightarrow \diam\left(\mathbf{S}^2, \gamma_0\right),
\end{equation}
and 
\begin{equation}\label{volumeconverge}
\vol(\mathbf{S}^2,\gamma_i)\rightarrow \vol(\mathbf{S}^2,\gamma_0).
\end{equation}

Now let $\Sigma_0$ and $\Sigma_i$ be the corresponding convex surfaces of $(\mathbf S^2,\gamma_0)$ and $(\mathbf S^2,\gamma_i)$ when isometrically embedded in $\mathbf R^3$. Due to \eqref{diameterconverge}, $\{\Sigma_i\}$ can be contained in a bounded ball after suitable translations. According to Blaschke's selection theorem (Theorem 1.8.7 in \cite{Sch}) and Lemma 1.8.1 in \cite{Sch}, $\{\Sigma_i\}$ has a subsequence that converges to the boundary of a compact convex set in the Hausdorff distance. Without loss of generality, we assume the subsequence is $\{\Sigma_i\}$ itself. Let $\Sigma_{\infty}$ be the Hausdorff limit of $\{\Sigma_i\}$. By Theorem 3 in \cite[p.424]{Alex}, $\Sigma_{\infty}$ is a point, a segment, or a generalized convex surface. Besides the usual convex surface, a generalized convex surface can be a doubly-covered planar convex  domain (two copies of a planar convex domain glued along the boundary).  If $\Sigma_{\infty}$ is a point or a segment, then $\Sigma_i$ can be enclosed by a convex surface with area goes to $0$ as $i\rightarrow 0$. By Lemma \ref{monoquermassintegral}, $|\Sigma_i|\rightarrow 0$, then we draw a contradiction from \eqref{volumeconverge}. Hence $\Sigma_{\infty}$ is a generalized  convex surface. Denote the intrinsic distances of $\Sigma_i$ and $\Sigma_{\infty}$ by $d_i$ and $d_{\infty}$ respectively. According to the convergence theory of convex surfaces in $\mathbf{R}^3$ (see P.13 in \cite{P}), $d_{i}$ uniformly converges to $d_{\infty}$. This means for any $\varepsilon>0$, there exists a $\delta>0$ and an integer $N$ such that for any $i\geq N$, any points $x_i,\,y_i\in \Sigma_i$ and $x,\,y\in\Sigma_{\infty}$ with $|x_i-x|<\delta$ and $|y_i-y|<\delta$, there holds 
$$\left|d_i(x_i,y_i)-d_{\infty}(x,y)\right|<\varepsilon.$$ Next, we show that
\begin{equation}\label{gh1}
\left(\Sigma_i,d_i\right)\xrightarrow{d_{\rm GH}} \left(\Sigma_{\infty},d_{\infty}\right).
\end{equation}
We discuss the following two cases: 

{\bf Case 1}. $\Sigma_{\infty}$ does not degenerate, i.e., $\Sigma_{\infty}$ is not a doubly-covered planar convex domain. Without loss of generality, we assume the origin is enclosed by $\Sigma_{\infty}$. Since $\Sigma_i\xrightarrow{d_{\rm H}}\Sigma_{\infty}$, where $d_{\rm H}$ denotes the Hausdorff distance, there exists a sequence of constants $\{\lambda_i\}$ that $\lambda_i\geq 1$ and $\lambda_i\rightarrow 1$ such that $\lambda_i^{-1}\Sigma_{\infty}$ is enclosed by $\Sigma_i$ and $\Sigma_i$ is enclosed by $\lambda_i\Sigma_{\infty}$.
Now, for each $i$, we define two maps 
\begin{equation*}
f_i: \Sigma_i\rightarrow \Sigma_{\infty}\ \ \ \ x\mapsto\lambda_i\mathcal P_{\Sigma_i,\,\lambda_i^{-1}\Sigma_{\infty}}(x),
\end{equation*}
and 
\begin{equation*}
h_i: \Sigma_{\infty}\rightarrow\Sigma_i\ \ \ \ x\mapsto\mathcal{P}_{\lambda_i\Sigma_{\infty},\,\Sigma_i}(\lambda_i x).
\end{equation*}
Here $\mathcal{P}_{\Sigma_i,\,\lambda_i^{-1}\Sigma_{\infty}}: \Sigma_i\rightarrow \lambda_i^{-1}\Sigma_{\infty}$ is the nearest point projection. Since $\Sigma_{\infty}$ is convex and $ \lambda_i^{-1}\Sigma_{\infty}$ is enclosed by $\Sigma_i$, this map is well-defined. $\mathcal{P}_{\lambda_i\Sigma_{\infty},\,\Sigma_i}:\lambda_i\Sigma_{\infty}\rightarrow\Sigma_i$ is defined in the same way.

From the facts that $d_i\rightarrow d_\infty$ uniformly and $\lambda_i\rightarrow1$, it is not hard to see that $f_i$ and $h_i$ are $\varepsilon_i$-GH approximations (for the definition of $\varepsilon$-GH approximation, we recommend \cite[P. 202]{R}) between $\Sigma_i$ and $\Sigma_{\infty}$ with $\varepsilon_{i}\rightarrow0$. Thus, by Lemma 1.3.4 in \cite{R}, \eqref{gh1} holds. 

{\bf Case 2}. $\Sigma_{\infty}$ is a doubly-covered planar convex domain. Denote the convex domain that generates $\Sigma_{\infty}$ by $D$. Without loss of generality, we assume $D$ lies in the $xy$-plane and the origin is an interior point of $D$. Designate one copy of $\Sigma_{\infty}$ as the ``upper" sheet for the convergence of points in the upper position of $\Sigma_i$. Similarly, we designate the other copy as the ``lower" sheet. Denote the upper sheet and lower sheet by $D\times\{0_+\}$ and $D\times\{0_-\}$ respectively. For $\varepsilon>0$, let ${\rm Cy}(D,\varepsilon)$ be the boundary of the convex body $D\times[-\varepsilon,\varepsilon]$. Since $\Sigma_i$ converges to $\Sigma_{\infty}$, there exists a sequence $\{\lambda_i\}$ that  $\lambda_i>1$ and $\lambda_i\rightarrow 1$ and a sequence $\{\delta_i\}$ that $\delta_i>0$ and $\delta_i\rightarrow 0$ such that $\lambda_i^{-1}\Sigma_{\infty}$ is enclosed by $\Sigma_i$ and $\Sigma_i$ is enclosed by ${\rm Cy}(\lambda_iD,\delta_i)$. Apparently, ${\rm Cy}(\lambda_iD,\delta_i)\xrightarrow{d_{\rm H}}\Sigma_{\infty}$ as $i\rightarrow\infty$. We first 
construct a map $f_i$ by
\begin{equation*}
f_i: \Sigma_i\rightarrow \Sigma_{\infty}\ \ \ \ x\mapsto\lambda_i\mathcal P_{\Sigma_i,\,\lambda_i^{-1}\Sigma_{\infty}}(x),
\end{equation*}
where the projection on $\lambda\Sigma_{\infty}$ is regarded as on $\lambda D\times\{0_+\}$, $\lambda D\times\{0_-\}$ or $\lambda(\partial D)$ depending on the projecting direction is downward, upward or horizontal. For each $i$, define a map $\Psi_i:\Sigma_{\infty}\rightarrow {\rm Cy}(\lambda_iD,\delta_i)$ by
\begin{equation*}
\left\{
\begin{aligned}
\Psi_i(x,0_{\pm})&=(\lambda x,\pm \delta_i)\quad\  x\in \mathring D, \\
\Psi_i(x)&=(\lambda x,0)\quad\quad\  x\in\partial D.
\end{aligned}
\right.
\end{equation*}
Then we construct a map $h_i$ by
 \begin{equation*}
h_i: \Sigma_{\infty}\rightarrow\Sigma_i\ \ \ \ x\mapsto\mathcal{P}_{{\rm Cy}(\lambda_iD,\,\delta_i),\,\Sigma_i}(\Psi_i(x)).
\end{equation*}
One can check that $f_i$ and $h_i$ are $\varepsilon_i$-GH approximations with $\varepsilon_i\rightarrow 0$ as $i\rightarrow\infty$. Hence, \eqref{gh1} also holds when $\Sigma_{\infty}$ is degenerate.

Combining \eqref{gh} and $\eqref{gh1}$ together, we conclude that 
$$d_{\rm GH}\left(\left(\mathbf{S}^2, d_{\gamma_0}\right),\left(\Sigma_{\infty},d_{\infty}\right)\right)=0.$$
Then, by Theorem 7.3.30 in \cite{BBI}, we know that
$\left(\mathbf{S}^2, d_{\gamma_0}\right)$ is isometric to $\left(\Sigma_{\infty},d_{\infty}\right)$. Here ``isometric" means that there exists a homomorphism $\phi:\mathbf{S}^2\rightarrow \Sigma_{\infty}$ such that $d_{\gamma_{0}}(x,y)=d_{\infty}(\phi(x),\phi(y))$ for any $x,y\in \mathbf{S}^2$.

Next, by Pogorelov's rigidity theorem \cite[P.167, Theorem 1]{P}, $\Sigma_{\infty}$ and $\Sigma_0$ are congruent, which means that there is an isometry of $\mathbf R^3$ that maps $\Sigma_{\infty}$ to $\Sigma_0$. Finally,  by Lemma \ref{monoquermassintegral}, we conclude that 
$$\lim_{i\rightarrow\infty}\Lambda_+\left(\mathbf S^2,\gamma_i\right)=\Lambda_+(\mathbf S^2,\gamma_0).$$
Thus we get a contradiction. The proof of this theorem is finished.
\end{proof}

We can weaken the Lipschitz distance in Theorem \ref{2dstability} to the Gromov-Hausdorff distance, namely
\begin{thm}\label{GHstability}
Let $\gamma_0$ be a smooth metric on $\mathbf S^2$ with $K_{\gamma_0}\geq0$. For any $\varepsilon>0$, there is a $\delta=\delta(\varepsilon,\gamma_0)>0$ such that for any
metric $\gamma$ with $K_\gamma\geq 0$ and $d_{\rm GH}\left((\mathbf{S}^2,\gamma),(\mathbf{S}^2,\gamma_0)\right)\leq\delta$,   
$$
\big|\Lambda_+\left(\mathbf S^2,\gamma\right)-\Lambda_+(\mathbf S^2,\gamma_0)\big|\leq\varepsilon.
$$
\end{thm}

Most of the proof of Theorem \ref{2dstability} also works in the situation of Gromov-Hausdorff convergence. The additional things to verify is that \eqref{diameterconverge} and \eqref{volumeconverge} still hold under the Gromov-Hausdorff convergence. The first one is apparent, while the second one is guaranteed by Theorem 0.1 in \cite{Co}.

\begin{rem}
Theorem \ref{GHstability} can be generalized to the metrics having a lower curvature bound $\kappa$ for any $\kappa\leq 0$, with $\Lambda_+$ replaced by $\Lambda_{+,\,\kappa}$ simultaneously.
\end{rem}

\begin{proof}[Proof of Corollary \ref{cby}] 
Due to the Gauss-Bonnet Theorem, $\partial \Omega$ is a $2$-sphere. Denote the induced metric on $\partial \Omega$ from $g$ by $\gamma$.  By the $C^0$-Limit Theorem in \cite{Gro14} or Theorem 1 in \cite{Bam}, the Gauss curvature of $\gamma$ is nonnegative. Hence, Corollary \ref{cby} follows immediately from Theorem \ref{2dstability}.
\end{proof}

\section{Estimates for $\Lambda_+(\mathbf{S}^n, \gamma)$ when $n\geq3$}\label{section6}

In this section, we give estimates for $\Lambda_+(\mathbf{S}^n, \gamma)$ when $n\geq3$. A key observation is that {\it we can construct a smooth path of PSC metrics connecting a fixed PSC metric and another PSC metric that is sufficiently close to the fixed one in the $C^0$-topology.} Our argument mainly relies on the Ricci-DeTurck flow and quasi-spherical equation. Let us begin with a brief overview of the Ricci-DeTurck flow.

Let $(M,\bar g_0)$ be a smooth $n$-dimensional closed Riemannian manifold. The Ricci flow on $M$ starting from $\bar g_0 $ is a smooth, time-dependent family of Riemannian metrics $\bar{g}(t)$ that solves the following equation
\begin{equation}\label{R}
\left\{
\begin{aligned}
\partial_t\bar{g}(t)&=-2\ric_{\bar{g}(t)},\\
\bar{g}(0)&= \bar g_0.
\end{aligned}
\right.
\end{equation}
It is well known that a short time solution to the Ricci flow always exists and is unique (e.g., see \cite{PTopping}). Now we consider the Ricci-DeTurck flow with a background metric. Let $X_{\bar{g}}$ be the operator which maps a metric tensor to a vector field defined by
\begin{equation}\label{lie}
X_{\bar{g}}(g)=\sum_{i=1}^{n}\left(\nabla_{e_i}^{\bar{g}}e_i-\nabla_{e_i}^{{g}}e_i\right),
\end{equation}
where $\{e_i\}_{i=1}^n$ is any local orthonormal basis with respect to $g$. The Ricci-DeTurck flow with background metric $\bar g$ (Here $\bar g$ is either a fixed metric or a Ricci flow) and initial metric $g_0$ is
\begin{equation}\label{rt}
\left\{
\begin{aligned}
\partial_tg(t)&=-2\ric_{g(t)}-\mathcal{L}_{X_{\bar{g}}\left(g(t)\right)}g(t),\\
g(0)&= g_0.
\end{aligned}
\right.
\end{equation}

If $g(t)$ solves \eqref{rt}, then one can obtain a Ricci flow via a family of diffeomorphisms. More precisely, letting $\chi(t)$ be the family of diffeomorphisms satisfying
\begin{equation}\label{diff}
\left\{
\begin{aligned}
\partial_t\chi(t)&=X_{\bar{g}}(g(t))\circ\chi(t)\\
\chi(0)&= id,
\end{aligned}
\right.
\end{equation}
then $\chi(t)^*g(t)$ is the solution of Ricci flow with initial data $g(0)$.

Note that under the Ricci-DeTurck flow, the scalar curvature satisfies the following evolution equation (see Equation 2.24 in \cite{BG}),
$$\partial_t R_{g(t)}\geq \Delta_{g(t)}R_{g(t)}-\left\langle X_{\bar{g}}(g(t)),\nabla_{g(t)} R_{g(t)}\right\rangle_{g(t)}+\frac{2}{n}R_{g(t)}^2.$$
If $R_{g(0)}\geq \kappa$, then by the maximum principle, we have
$$R_{g(t)}\geq \frac{\kappa}{1-\frac{2\kappa}{n}t}.$$
Hence, positivity of scalar curvature is preserved under the Ricci-DeTurck flow.

In \cite{Simon}, Simon studied the existence and higher order derivatives estimates for the Ricci-DeTurck flow on a fixed background metric. If we have some further assumptions on the initial metric $g_0$, we may get better estimates for this flow.

\begin{lm}[Lemma 3.2 and Lemma 3.3 in \cite {ST18}]\label{1223}
For any $p>n$, there exists a constant $\epsilon_1(n,p)>0$ such that the following holds: Let $(M,\bar g)$ be a smooth closed $n$-dimensional Riemannian manifold. There exists a constant $T>0$ depending only on $n$ and $\|\Rm_{\bar g}\|_{L^\infty(M)}$, such that for any constant $G>0$ and any smooth metric $g_0$ on $M$ with $\|g_0-\bar g\|_{L^\infty(M)}\leq\epsilon_1(n,p)$ and $\left\|\nabla g_0\right\|_{L^p(M)}\leq G$, the Ricci-DeTurck flow \eqref{rt} with initial metric $g_0$ and background metric $\bar g$ admits a solution $g(t)$ in $[0,T]$ that satisfies
$\frac{1}{2}\bar g\leq g(t)\leq 2\bar g$ and
$$\left\|\nabla g(t)\right\|_{L^\infty(M)}\leq Ct^{-\frac{\sigma}{2}}, \quad\  \left\|\nabla^2 g(t)\right\|_{L^\infty(M)}\leq Ct^{-\frac{1+\sigma}{2}}$$
for all $t\in(0,T]$, where $\sigma=n/p$, the constant $C$ depends only on $n$, $p$, $G$ and $\bar{g}$. Here, the norms, covariant derivatives and integration are taken with respect to $\bar g$. 
\end{lm}

In \cite{BG}, Burkhardt-Guim studied the Ricci-DeTurck flow on a Ricci flow background and obtained the following result:
\begin{lm}[see Lemma 3.3, Corollary 3.4 in \cite{BG}]\label{exist}
There exists a dimensional constant $\epsilon_2(n)>0$ such that the following holds: Let $M$ be a smooth closed $n$-dimensional manifold, $\bar g_0$ be a smooth metric on $M$ and $\bar g(t)$ be the Ricci flow starting from $\bar g_0$ in a time interval $[0,T']$. There exists $T\in (0,T']$ depending only on $\{\bar g(t)\}_{t\in [0,T']}$ and constants
 $$c_k=c_k\left(n,\, k,\,\sup_{t\in [0,T],\, l\leq k}\left\|\nabla^{l}\Rm_{\bar{g}(t)}\right\|_{L^{\infty}(M)}\right)$$
 such that for every smooth metric $g_0$ on $M$ with $\|g_0-\bar{g}_0\|_{L^{\infty}(M)}<\epsilon_2(n)$, the Ricci-DeTurck flow with initial metric $g_0$ and background Ricci flow $\{\bar g(t)\}_{t\in[0,T']}$ admits a smooth solution $g(t)$ in $[0,T]$ that satisfies
\begin{equation}\label{ma}
\left\|\nabla^k\left(g(t)-\bar{g}(t)\right)\right\|_{L^{\infty}(M)}\leq c_kt^{-\frac{k}{2}}\left\|g_0-\bar{g}_0\right\|_{L^{\infty}(M)}
\end{equation}
for all $t\in(0,T]$. Here the covariant derivatives are taken with respect to $\bar{g}(t)$ and the norms are taken with respect to $\bar{g}_0$. 
\end{lm}

For a closed manifold $\Sigma$, by employing the Ricci-DeTurck flow, we can prove the following $C^0$-local path-connectedness of $\mathcal M_{\rm psc}\left(\Sigma\right)$, which may have its own interest. 
 \begin{thm}\label{localconnectednesspsc}
Let $\Sigma$ be a smooth closed $n$-dimensional manifold and $\gamma_0$ be a smooth PSC metric on $\Sigma$. Then there exists a positive constant $\delta=\delta(n,\gamma_0)$ such that for any smooth PSC metric $\gamma$ on $\Sigma$ with
 $$\|\gamma-\gamma_0\|_{L^\infty(\Sigma)}\leq \delta,$$
 there is a smooth path $\{\eta(t)\}_{t\in[0,1]}$ in $\mathcal M_{\rm psc}\left(\Sigma\right)$ that connects $\gamma$ and $\gamma_0$, and satisfies
 $$
\big\|\eta'(t)\big\|_{L^\infty(\Sigma)}\leq Ct^{-1}\quad\, \text{for}\ \, t\in (0,1],
$$
where $C$ depends only on $n$ and the curvature bound of $\gamma_0$. Here and in the proof, the norms are taken with resect to $\gamma_0$.
 \end{thm}

\begin{proof}
Let $\bar\gamma(t)$ be the Ricci flow starting from $\gamma_0$. By Lemma \ref{exist}, there exists some $T>0$ such that for any metric $\gamma$ with $\|\gamma-\gamma_0\|_{L^\infty(\Sigma)}<\epsilon_2(n)$, the Ricci-DeTurck flow with initial metric $\gamma$ and background Ricci flow $\{\bar\gamma(t)\}_{t\in[0,T]}$ admits a smooth solution $\gamma(t)$ in $[0,T]$ that satisfies
\begin{equation}\label{pathesitmate}
\left\|\nabla^k(\gamma(t)-\bar\gamma(t))\right\|_{L^\infty(\Sigma)}\leq c_k t^{-\frac{k}{2}}\|\gamma-\gamma_0\|_{L^\infty(\Sigma)},
\end{equation}
where the covariant derivatives are taken with respect to $\bar\gamma(t)$ and $c_k$ depends only on the Ricci flow $\{\bar\gamma(t)\}_{t\in[0,T]}$ for $k=0$, $1$, $2$. Without loss of generality, we assume $T\leq1$. From Equations (2.8), (2.9), (2.11), (2.13) and (2.17) in \cite{BG}, we see that
\begin{align*}
\left\|\partial_t\gamma(t)\right\|\leq\, &C_1\left(\left\|\nabla^2\left(\gamma(t)-\bar\gamma(t)\right)\right\|+\left\|\nabla\left(\gamma(t)-\bar\gamma(t)\right)\right\|^2\right)\\
&+(C_1+1)\left\|\Rm_{\bar{\gamma}(t)}\right\|.
\end{align*}
where $C_1$ depends only on $n$ if $c_0\|\gamma-\gamma_0\|_{L^\infty(\Sigma)}$ is sufficiently small. Note that $\left\|\Rm_{\bar{\gamma}(t)}\right\|$ is controlled by the  curvature bound of $\gamma_0$ in $[0,T]$. Therefore, there exist a constant $\epsilon_3=\epsilon_3(n,c_0,c_1,c_2)>0$ and a constant $C>0$ depending only on $n$ and the curvature bound of $\gamma_0$ such that if $\|\gamma-\gamma_0\|_{L^\infty(\Sigma)}\leq\epsilon_3$, then $\left\|\partial_t\gamma(t)\right\|\leq Ct^{-1}$ for $t\in (0,T]$.

By estimate \eqref{pathesitmate}, if $\|\gamma-\gamma_0\|_{L^\infty(\Sigma)}$ is small enough, then the segment $\{(1-t)\gamma(T)+t\bar\gamma(T)\}_{t\in[0,1]}$ is a path of PSC metrics. Choose a sufficiently small $\delta$ that satisfies above requirements. Let
\begin{align*}
\bar\eta(t)=
\left\{
\begin{aligned}
&\gamma\left(3Tt\right)\qquad\qquad\qquad\qquad\qquad \ \ \ \,\text{for}\ \ \,0\leq t\leq \frac{1}{3},\\
&\left(2-3t\right)\gamma(T)+\left(3t-1\right)\bar\gamma(T)\quad\ \ \,  \text{for}\ \ \, \frac{1}{3}< t<\frac{2}{3},\\
&\bar\gamma\left(3T-3Tt\right) \qquad\qquad\qquad\qquad \ \ \, \text{for}\ \ \,\frac{2}{3}\leq t\leq 1.\\
\end{aligned}
\right.
\end{align*}
Apparently, $\{\bar\eta(t)\}_{t\in [0,1]}$ is a piecewise smooth path in $\mathcal M_{\rm psc}\left(\Sigma\right)$ that connects $\gamma$ and $\gamma_0$. In $[0,1/3]$, $\|\bar\eta'(t)\|\leq Ct^{-1}$. In $[1/3,2/3]$, $\|\bar\eta'(t)\|\leq 3\|\gamma(T)-\bar\gamma(T)\|$. In $[2/3,1]$, $\|\bar\eta'(t)\|\leq 3\|\partial_t\bar\gamma\|_{L^\infty(\Sigma\times[0,T])}$.  Finally, by using the mollifying method in \cite{CM} (see Proposition 2.1), we can get a smooth path $\{\eta(t)\}_{t\in [0,1]}$ that meets all the requirements.
\end{proof}

The main result in this section is the following:

\begin{thm}\label{continuity1-re}
Let $2\leq n\leq 6$. Given constants $G>0$ and $p>n$, for any $\varepsilon>0$, there is a $\delta(n,p,G,\varepsilon)>0$ such that for any $\gamma\in\mathcal M^n_{p,\,G}$ with $\|\gamma-\gamma_{\rm std}\|_{L^\infty(\mathbf{S}^{n},\gamma_{\rm std})}\leq\delta$, there holds 
\begin{equation*}
\left |\Lambda_+(\mathbf{S}^{n},\gamma)-\Lambda_+(\mathbf{S}^{n},\gamma_{\rm std})\right |\leq\varepsilon.
\end{equation*}	
\end{thm}

\begin{proof}
Let $\bar\gamma(t)$ be the Ricci flow starts from $\gamma_{\rm std}$ in $[0,1/4n]$. Then $\bar\gamma(t)=\left(1-2(n-1)t\right)\gamma_{\rm std}$. It is not hard to see from \eqref{lie} that $X_{\bar\gamma(t)}(\gamma)=X_{\gamma_{\rm std}}(\gamma)$ for any smooth metric $\gamma$ on $\mathbf S^{n}$. Hence, the Ricci-DeTurck flow on fixed background metric $\gamma_{\rm std}$ and Ricci flow $\bar\gamma(t)$ coincide. By Lemma \ref{1223} and Lemma \ref{exist}, for any metric $\gamma$ with $\|\gamma-\gamma_{\rm std}\|\leq\min\{\epsilon_1(n,p),\epsilon_2(n)\}$, where $\epsilon_1(n,p)$ and $\epsilon_2(n)$ are the constants in Lemma \ref{1223} and Lemma \ref{exist}, there is a solution $\gamma(t)$ in $[0,T(n)]$ to the Ricci-DeTurck flow with initial metric $\gamma$ and above two types of background. Here we make a convention that $C$ is a constant depending only on $n$, $p$ and $G$, which may vary from line to line. Since $\|\nabla_{\gamma_{\rm std}}\gamma\|_{L^p(\mathbf{S}^{n},\gamma_{\rm std})}\leq G$, by Lemma \ref{1223}, we have 
\begin{equation*}
\left\|\nabla_{\gamma_{\rm std}}\gamma(t)\right\|_{\gamma_{\rm std}}\leq Ct^{-\frac{\sigma}{2}}, \quad \text{and} \quad \left\|\nabla^2_{\gamma_{\rm std}}\gamma(t)\right\|_{\gamma_{\rm std}}\leq Ct^{-\frac{1+\sigma}{2}},
\end{equation*}
where $\sigma=n/p$. Then it follows from the evolution equation of $\gamma(t)$ (see (1.5) in \cite{Simon}) that
\begin{equation*}
\left\|\partial_t\gamma(t)\right\|_{\gamma_{\rm std}}\leq Ct^{-\frac{1+\sigma}{2}}\quad\,\text{in}\ \left(0,T(n)\right].
\end{equation*}
The rest of the proof is divided into the following two steps.\medskip \\ 
{\it Step 1:} Set $N=[\frac{2}{1-\sigma}]+1$. Let $s_0\in (0,T(n)]$ be a constant to be determined later. Choose $\tilde\varepsilon$ that $\left(\left(1+\tilde\varepsilon\right)^{n-1}-1\right)\Lambda_+(\mathbf{S}^{n},\gamma_{\rm std})=\varepsilon$. On $[0,s_0]$, let $B(s)=1+\tilde \varepsilon s^{-1}_0s$ and $\gamma_1(s)=B(s)\gamma(s^N)$. Then $\{\gamma_1(s)\}_{s\in [0,s_0]}$ is a smooth path of PSC metrics. Simple calculations give
\begin{align*}
\gamma_1'(s)&=\tilde\varepsilon s^{-1}_0\gamma\left(s^N\right)+NB(s)s^{N-1}\gamma'\left(s^N\right)\\
&\geq\left(\tilde\varepsilon s^{-1}_0-2\left(1+\tilde\varepsilon\right)Ns^{N-1}\left\|\gamma'\left(s^N\right)\right\|_{\gamma_{\rm std}}\right)\gamma\left(s^N\right)\\
&\geq s^{-1}_0\left(\tilde\varepsilon -C\left(1+\tilde\varepsilon\right)s_0^{\frac{N(1-\sigma)}{2}}\right)\gamma\left(s^N\right).
\end{align*}
Thus we can choose $s_0=s_0(n,p,G,\tilde\varepsilon)$ sufficiently small so that $\gamma_1(s)$ monotonically increases in $[0,s_0]$. Note that $\{\gamma_1(s)\}_{s\in[0,s_0]}$ connects $\gamma$ and $\left(1+\tilde\varepsilon\right)\gamma\left(s_0^N\right)$. Then according to Lemma \ref{c0monotonicity}, we have
\begin{equation}\label{c11}
\Lambda_+(\mathbf{S}^{n},\gamma)\leq\Lambda_+\left(\mathbf{S}^{n},\left(1+\tilde\varepsilon\right)\gamma\left(s_0^N\right)\right).
\end{equation}
\medskip
{\it Step 2:} Let $\gamma_2(s)=(1+\tilde\varepsilon)s\bar\gamma\left(s_0^N\right)+(1-s)\gamma\left(s_0^{N}\right)$ for $s\in [0,1]$. Lemma \ref{exist} tells us that
$$\left\|\nabla^k_{\bar\gamma(s_0^N)}\left(\gamma\left(s_0^N\right)-\bar\gamma\left(s_0^N\right)\right)\right\|_{L^{\infty}(\mathbf{S}^{n},\gamma_{\rm std})}\leq c_ks_0^{-\frac{kN}{2}}\|\gamma-{\gamma}_{\rm std}\|_{L^{\infty}(\mathbf{S}^{n},\gamma_{\rm std})},$$
where $c_k$ depends only on $k$ and $n$ for $k=0$, $1$, $2$.  
Hence, there exists some $\delta=\delta(n,p,G,\tilde\varepsilon)>0$ such that if $\|\gamma-\gamma_{\rm std}\|_{L^{\infty}(\mathbf{S}^{n})}\leq \delta$, $\{\gamma_2(s)\}_{s\in[0,1]}$ is a smooth monotonically increasing path of PSC metrics. By Lemma \ref{c0monotonicity}, we have
\begin{equation}\label{c2}
\begin{split}
\Lambda_+\left(\mathbf{S}^{n},\gamma\left(s_0^N\right)\right)&\leq \Lambda_+\left(\mathbf{S}^{n},(1+\tilde\varepsilon)\bar\gamma\left(s_0^N\right)\right)\\
&= (1+\tilde\varepsilon)^{\frac{n-1}{2}}\left[1-2(n-1)s^N_0\right]^{\frac{n-1}{2}}\Lambda_+\left(\mathbf{S}^{n},\gamma_{\rm std}\right).
\end{split}
\end{equation}
Combining \eqref{c11} and \eqref{c2}, we arrive at 
$$\Lambda_+\left(\mathbf{S}^{n},\gamma\right)\leq\left(1+\tilde\varepsilon\right)^{n-1}\Lambda_+\left(\mathbf{S}^{n},\gamma_{\rm std}\right).$$
Via a very similar approach, we can construct two ``reverse" smooth monotonically increasing paths of PSC metrics, and obtain the reverse inequality
$$\Lambda_+\left(\mathbf{S}^{n},\gamma_{\rm std}\right)\leq\left(1+\tilde\varepsilon\right)^{n-1}\Lambda_+\left(\mathbf{S}^{n},\gamma\right).$$
From the choice of $\tilde\varepsilon$, we see the conclusion holds.
\end{proof}

By the convergence theory of Riemannian manifolds, we know that certain classes of metrics possess uniform $W^{1,\,p}$-boundedness up to diffeomorphisms. And by definition, $\Lambda_+$ is invariant under diffeomorphisms. The following Theorem \ref{main} can be regarded as  an application of Theorem \ref{continuity1-re}. In the following, all norms and covariant derivatives are taken with respect to $\gamma_{\rm std}$. Given positive constants $K$ and $i_0$, define
\begin{equation}\label{N}
 \mathcal N_n\left(K, i_0\right)=\left\{\gamma\in\mathcal M_{\rm psc}(\mathbf S^{n}) \,\big |\, \ric_\gamma\geq -K,\, \inj_\gamma\geq i_0\right\},
\end{equation}
where $\inj_\gamma$ denotes the injective radius of $(\mathbf S^{n},\gamma)$. 

\begin{thm}\label{main}
Let $2\leq n\leq 6$. Given $K>0$ and $i_0>0$, for any $\varepsilon>0$, there exists a constant $\delta=C(n,K, i_0,\varepsilon)>1$ such that for any $\gamma\in\mathcal N_n\left(K, i_0\right)$ with $\dial(\gamma)\leq\delta$, 
\begin{equation*}
\left |\Lambda_+(\mathbf{S}^{n},\gamma)-\Lambda_+(\mathbf{S}^{n},\gamma_{\rm std})\right |\leq\varepsilon.
\end{equation*}	
\end{thm}
To prove this result, we need the following lemma.
\begin{lm}\label{keylem}
Let $\gamma$ be a $W^{1,\,p}$ $(p>n)$ metric on $\mathbf S^n$ with $\dial(\gamma)=1$. Then we can find a $W^{2,\,p}$ homeomorphism $\psi$ such that $\gamma=\psi^*(\gamma_{\rm std})$, and 
\begin{equation*}
\left\|D\psi\|_{W^{1,\,p}}\leq C\left(n,\|\gamma\right\|_{W^{1,\,p}}\right).
\end{equation*}
\end{lm}

\begin{proof}
By definition, there exists a sequence of diffeomorphisms $\{\psi_i\}$ such that
\begin{equation}\label{uselessassu}
\frac{i}{i+1}\psi_i^*(\gamma_{\rm std})\leq \gamma\leq \frac{i+1}{i}\psi_i^*(\gamma_{\rm std}).
\end{equation}
Denote the inverse of $\psi_i$ by $\phi_i$.
By \eqref{uselessassu}, we have
\begin{equation*}
\|D\psi_i\|_{L^{\infty}}\leq \frac{i+1}{i}\|\gamma\|_{L^{\infty}},\quad\,\text{and}\quad\, \|D\varphi_i\|_{L^{\infty}}\leq \frac{i+1}{i}\|\gamma^{-1}\|_{L^{\infty}}.
\end{equation*}
Since $\{\psi_i\}$ and $\{\phi_i\}$ have uniform Lipschitz constants, by taking subsequence if necessary, we may assume $\psi_i$ and $\varphi_i$ uniformly converge to Lipschitz maps $\psi$ and $\varphi$ respectively. Clearly, $\psi\circ\varphi=id$. So $\psi$ and $\varphi$ are Lipschitz homeomorphisms.
Inequality \eqref{uselessassu} then implies
\begin{equation*}
\frac{i}{i+1}d_{\gamma_{\rm std}}(\psi_i(x),\psi_i(y))\leq d_{\gamma}(x,y)\leq  \frac{i+1}{i}d_{\gamma_{\rm std}}(\psi_i(x),\psi_i(y)).
\end{equation*}
Letting $i\rightarrow\infty$, we arrive at
\begin{equation*}
d_{\gamma}(x,y)=d_{\gamma_{\rm std}}(\psi(x),\psi(y)).
\end{equation*}

Finally, from the proof of Theorem 2.1 in \cite{Ta}, it is not hard to see that $\psi\in W^{2,\,p}$ and $\|D\psi\|_{W^{1,\,p}}\leq C(n,\|\gamma\|_{W^{1,\,p}})$.

 \end{proof}

Now, we are ready to give the proof of Theorem \ref{main}.
\begin{proof}[Proof of Theorem \ref{main}]

We take a contradiction argument. If the conclusion is not true, then there exists a constant $\varepsilon_0>0$ and a sequence of metrics $\{\gamma_i\} \in \mathcal N_n\left(K, i_0\right)$ with $\dial(\gamma_i)\rightarrow 1$ such that
$$\left |\Lambda_+(\mathbf{S}^{n},\gamma_i)-\Lambda_+(\mathbf{S}^{n},\gamma_{\rm std})\right |\geq\varepsilon_0.$$

Fixing a constant $p>n$, due to Anderson-Cheeger's compactness result (see \cite{AC}), there is a subsequence of $\{\gamma_i\}$ (for simplicity we assume the subsequence is $\{\gamma_i\}$ itself) and corresponding diffeomorphisms $\{\phi_i\}$ such that
$\phi_i^*(\gamma_i)$ converges to a $W^{1,\,p}$ metric $\gamma_{\infty}$ in the $W^{1,\,p}$-topology.

By the assumption $\lim_{i\rightarrow\infty}\dial(\gamma_i)=1$, we can find differeomorphisms $\psi_i$ and constants $\lambda_i\rightarrow 1$ such that
\begin{equation}\label{asu}
\lambda^{-1}_i\psi_i^*(\gamma_{\rm std})\leq \phi_i^*(\gamma_i)\leq\lambda_i\psi_i^*(\gamma_{\rm std}).
\end{equation}
Since $\phi_i^*(\gamma_i)$ converges to $\gamma_{\infty}$ in the $W^{1,\,p}$-topology, there exist a sequence of constants $\mu_i\rightarrow 1$ such that
\begin{equation}\label{con}
\mu_i^{-1}\gamma_{\infty}\leq\phi_i^*(\gamma_i)\leq\mu_i\gamma_{\infty}.
\end{equation}
It follows from \eqref{asu} and \eqref{con} that
$$(\lambda_i\mu_i)^{-1}\gamma_{\rm std}\leq\left(\psi_i^{-1}\right)^*(\gamma_{\infty})\leq\lambda_i\mu_i\gamma_{\rm std}.$$
Since $\lambda_i\mu_i\rightarrow 1$, we conclude that $\dial(\gamma_{\infty})=1$. By Lemma \ref{keylem}, there is a $W^{2,\,p}$ homeomorphism $\psi$ such that $\gamma_{\infty}=\psi^*(\gamma_{\rm std})$ and $\|D\psi\|_{W^{1,\,p}}\leq C(n,\|\gamma_{\infty}\|_{W^{1,\,p}})$. Apparently, $\psi^{-1}$ is also a $W^{2,\,p}$ homeomorphism and $(\psi^{-1})^*\phi_i^*(\gamma_i)$ converges to $\gamma_{\rm std}$ in the $W^{1,\,p}$-topology. 

Note that we can find a diffeomorphism $\tilde\psi$ by mollifying $\psi$ such that $\|D(\tilde\psi^{-1})-D(\psi^{-1})\|_{W^{1,\,p}}$ is arbitrarily small.
Then by Theorem \ref{continuity1-re}, we have
 $$ \lim_{i\rightarrow \infty}\Lambda_+(\mathbf{S}^{n},\gamma_i)=\Lambda_+(\mathbf{S}^{n},\gamma_{\rm std}).$$
Thus we get a contradiction. This completes the proof.
 \end{proof}

\end{document}